\documentclass[a4paper,10pt,psamsfonts]{amsart}

\usepackage[utf8]{inputenc} 

\usepackage{amsmath}
\usepackage{amstext}
\usepackage{amsfonts}
\usepackage{amsthm}
\usepackage{amssymb}
\usepackage{graphicx}
\usepackage[bf,SL,BF]{subfigure}
\usepackage{caption}
\usepackage{listings}
\usepackage{hyperref}
\hypersetup{colorlinks}

\usepackage{color}

\usepackage{changes}

\newcommand{\R}{{\mathbb{R}}}

\newcommand{\E}{{\mathbb{E}}}

\newcommand{\N}{{\mathbb{N}}}

\newcommand{\F}{{\mathcal{F}}} 
\renewcommand{\P}{{\mathbb{P}}} 

\newcommand{\B}{{\mathcal{B}}}

\newcommand{\diff}[1]{\,\mathrm{d}#1}

\newcommand{\one}{\mathbb{I}}
\newcommand{\ee}{\mathrm{e}}

\newcommand{\triple}{{\vert\kern-0.25ex\vert\kern-0.25ex\vert}}
\newcommand{\RM}{{\mathrm{RM}}}
\newcommand{\SRM}{{\mathrm{SRM}}}

\theoremstyle{plain}

\newtheorem{definition}{Definition}[section]
\newtheorem{theorem}[definition]{Theorem}
\newtheorem{lemma}[definition]{Lemma}

\newtheorem{prop}[definition]{Proposition}
\newtheorem{assumption}[definition]{Assumption}

\theoremstyle{definition}
\newtheorem{remark}[definition]{Remark}

\begin{document}

\title[Two quadrature rules for stochastic integrals]
{
Two quadrature rules for stochastic It\^{o}-integrals with fractional 
Sobolev regularity 
}

\author[M.~Eisenmann]{Monika Eisenmann}
\address{Monika Eisenmann\\
  Technische Universit\"at Berlin\\
  Institut f\"ur Mathematik, Secr. MA 5-3\\
  Stra\ss e des 17.~Juni 136\\
  DE-10623 Berlin\\
  Germany}
\email{m.eisenmann@tu-berlin.de}

\author[R.~Kruse]{Raphael Kruse}
\address{Raphael Kruse\\
Technische Universit\"at Berlin\\
Institut f\"ur Mathematik, Secr. MA 5-3\\
Stra\ss e des 17.~Juni 136\\
DE-10623 Berlin\\
Germany}
\email{kruse@math.tu-berlin.de}

\keywords{stochastic integration, quadrature rules, fractional
Sobolev spaces, Sobolev--Slobodeckij norm} 
\subjclass[2010]{60H05, 60H35, 65C30}  

\begin{abstract}
  In this paper we study the numerical quadrature of a stochastic integral,
  where the temporal regularity of the integrand is measured in the fractional
  Sobolev--Slobodeckij norm in $W^{\sigma,p}(0,T)$, $\sigma \in (0,2)$, $p \in
  [2,\infty)$. We introduce two quadrature rules: The first is best suited for
  the parameter range $\sigma \in (0,1)$ and consists of a
  Riemann--Maruyama approximation on a randomly shifted grid. The second
  quadrature rule considered in this paper applies to the case of a  
  deterministic integrand of fractional Sobolev regularity with $\sigma \in
  (1,2)$. In both cases the order of convergence is equal to $\sigma$ with
  respect to the $L^p$-norm. 
  As an application, we consider the stochastic integration
  of a Poisson process, which has discontinuous sample paths. The theoretical
  results are accompanied by numerical experiments.
\end{abstract}

\maketitle

\section{Introduction}
\label{sec:intro}

In this paper we investigate the quadrature of stochastic It\^{o}-integrals. 
Such
quadrature rules are, for instance, important building blocks in numerical
algorithms for the approximation of stochastic differential equations (SDEs). 
For example, let $T \in (0,\infty)$ and $(\Omega_{W}, \F^W, (\F_t^W)_{t \in 
[0,T]},\P_W)$ be a filtered probability space satisfying the usual conditions. 
By $W \colon [0,T] \times \Omega_W \to \R$ we denote a standard
$(\F_t^W)_{t \in [0,T]}$-Wiener process.
Then, for a given continuous coefficient function $\lambda \colon [0,T] \to \R$
and a stochastically integrable process 
$G \colon [0,T] \times \Omega_W \to \R$ the numerical solution of the initial
value problem 
\begin{align*}
  \begin{cases}
    \diff{X(t)} = \lambda(t) X(t) \diff{t} + G(t) \diff{W(t)},
    \quad t \in [0,T],\\
    X(0) = 0,\\
  \end{cases}
\end{align*}
can be reduced to the quadrature of the It\^{o}-integral
\begin{align*}
  X(t) = \int_0^t \exp\Big( \int_s^t\lambda (u) \diff{u} \Big)
  G(s) \diff{W(s)}, \quad t \in [0,T],
\end{align*}
by the variation of constants formula. We refer to
\cite[Section~4.4]{kloeden1999} for further examples of SDEs which can be
reduced to quadrature problems. 

In the standard literature, as for example in
\cite{daun2017,heinrich2017,przybylowicz2009,przybylowicz2010,przybylowicz2015a, 
wasilkowski2001},
the regularity of the integrand is often measured in terms of H\"older norms.
However, in many cases
the order of convergence observed in numerical experiments is larger
than the theoretical order derived from the H\"older regularity. The starting
point of this paper is the observation that the gap between the 
theoretical and the experimental order of convergence can often be
closed if the regularity of the integrand is measured in terms of
fractional Sobolev spaces.

We then introduce two quadrature formulas: The first is a
Riemann--Maruyama quadrature rule but with a randomly 
shifted mesh. The second is a stochastic version of the trapezoidal
rule and is applicable to It\^{o}-integrals with deterministic integrands
possessing a higher order Sobolev regularity.  
As our main result we obtain error estimates with 
positive convergence rates even in the case of possibly 
discontinuous integrands. 

To give a more precise outline of this paper,
let $G \colon [0,T] \times \Omega_W \to \R$ be 
a stochastically integrable process as above.
We want to find a numerical approximation of the definite stochastic 
It\^{o}-integral
\begin{align}
  \label{eq:stochint}
  I[G] = \int_0^T G(s) \diff{W(s)}.
\end{align}
If $G \in C^{\gamma}([0,T];L^p(\Omega_W))$, $\gamma \in (0,1)$, $ p \in
[2,\infty)$, then one often applies the classical Riemann--Maruyama-type
quadrature formula  
\begin{align}
  \label{eq:RM}
  Q_N^\RM[G] = \sum_{j = 1}^N G(t_{j-1}) \big( W(t_j) - W(t_{j-1})
  \big), 
\end{align}
for the approximation of the stochastic integral $I[G]$,
where $N \in \N$ determines the equidistant step size $h = \frac{T}{N}$ and 
an equidistant partition of $[0,T]$ of the form
\begin{align}
  \label{eq:grid}
  \pi_h = \{ t_j := j h\, : \; j = 0,1,\ldots,N\} \subset [0,T].
\end{align}
Then, 
standard results in the literature, see for instance \cite{daun2017,
przybylowicz2015a,wasilkowski2001}, show that
\begin{align}
  \label{eq:HoelderEst}
  \big\| I[G] - Q_N^\RM[G] \big\|_{L^p(\Omega_W)} 
  \le C \|G \|_{C^\gamma( [0,T]; L^p(\Omega_W))} h^{\gamma}
\end{align}
for all $N \in \N$, where the constant $C$ is independent of $N$ and $h$. 

In this paper, we first focus on the case that the integrand 
$G \colon [0,T] \times \Omega_W \to \R$ is of lower temporal
regularity. To be more precise, we assume that
$G \in L^p(\Omega_W; W^{\sigma,p}(0,T))$ with $\sigma \in (0,1)$ and $p \in 
[2,\infty)$. See Equation~\eqref{eq:normWsp1} and \eqref{eq:normWsp2} 
below for the definition of the Sobolev--Slobodeckij norm. We emphasize 
that the space $W^{\sigma,p}(0,T)$
contains possibly discontinuous trajectories if $\sigma p < 1$. In particular,
several of the singular functions studied in \cite{przybylowicz2010} are
included in the fractional Sobolev spaces in a natural way.

In this situation we introduce a \emph{randomly shifted} version of the
Riemann--Maru\-yama quadrature rule \eqref{eq:RM} for the approximation of
\eqref{eq:stochint}. 
To this end, let $N \in \N$ and set $h = \frac{T}{N}$ as above. 
We will, however, not make use of the equidistant partition \eqref{eq:grid}.
Instead we introduce an additional probability space 
$(\Omega_{\Theta}, \F^{\Theta}, \P_{\Theta})$ as 
well as a uniformly distributed random variable 
$\Theta \colon \Omega_{\Theta} \to [0,1]$, that is assumed to be 
independent of the
stochastic processes $G$ and $W$ in \eqref{eq:stochint}. The value of 
$\Theta$ then determines a randomly shifted equidistant partition 
$\pi_h(\Theta)$ of $[0,T]$ defined by 
\begin{align}
  \label{eq:grid2}
  \pi_h(\Theta) = \{ 0 \} \cup \{ \Theta_j := (j - 1 + \Theta) h \, : \; j = 1,\ldots,N\}
  \cup \{T\} \subset [0,T],
\end{align}
where we also write $\Theta_0 := 0$ and $\Theta_{N+1} := T$.
Note that $\pi_h(\Theta)$ is strictly speaking not equidistant due to the 
addition
of the initial and final time point. However, it holds true that
\begin{align}
  \label{eq:dist}
  | \Theta_j - \Theta_{j-1}| \le h
\end{align}
for all $j \in \{1,\ldots,N+1\}$, where we have equality in \eqref{eq:dist} for
all $j \in \{2,\ldots,N\}$. The \emph{randomly shifted Riemann--Maruyama
quadrature rule} is then given by 
\begin{align}
  \label{eq:Q}
  Q_N^{\SRM}[G,\Theta] = \sum_{j = 1}^N G(\Theta_j) \big( W(\Theta_{j+1}) - 
  W(\Theta_{j})
  \big). 
\end{align}
In Section~\ref{sec:randQuad} we will show that
$Q^{\SRM}_N$ is well-defined for all progressively measurable $G
\in L^p(\Omega_W;W^{\sigma,p}(0,T))$. If $G$ satisfies an additional 
integrability condition at $t = 0$ we have 
\begin{align*}
  \big\| I[G] - Q_N^\SRM[G,\Theta] \big\|_{L^p(\Omega_W \times 
  \Omega_{\Theta})} 
  \le C( 1 +  \| G \|_{L^p(\Omega_W;W^{\sigma,p}(0,T))}) h^{\sigma},
\end{align*}
where $C \in (0,\infty)$ is a suitable constant independent of $N$ and $h$. For
a precise statement of our conditions on $G$ we refer to Assumption~\ref{as:G}
below.

We remark that quadrature formulas for stochastic integrals 
on random time grids are already studied in the literature. In contrast 
to our observation, however, it usually turns out that the additional
randomization does not yield any advantage over algorithms with deterministic
grid points if the regularity of the integrand is measured in terms 
of the H\"older norm. See, for instance, \cite{daun2017}. We also refer to
\cite{geiss2006} for a related observation in mathematical finance.  

In Section~\ref{sec:higherOrder} we further discuss the case of deterministic
integrands $g \colon [0,T] \to \R$ with regularity $g \in W^{1+\sigma,p}(0,T)$,
$\sigma \in (0,1)$, $p \in [2,\infty)$. Under this additional 
regularity assumption we obtain a higher order error estimate for a
stochastic version of a generalized trapezoidal quadrature rule 
given by
\begin{align}
  \label{eq:Trap}
  \begin{split}
    Q^{\mathrm{Trap}}_N [g] 
    &= \sum_{j=1}^{N} \frac{1}{2} (g(\theta_j) + g(\hat{\theta}_j)) (W(t_j)- 
    W(t_{j-1}))\\
    &\quad + \sum_{j=1}^{N} \frac{1}{h} (g(t_j) - g(t_{j-1}) ) 
    \int_{t_{j-1}}^{t_j} (t - t_{j -\frac{1}{2}} ) \diff{W(t)},
  \end{split}
\end{align}
where $t_{j - \frac{1}{2}} = \frac{1}{2} (t_{j-1} + t_j)$, $\theta_j = t_{j-1} + 
\theta h$ and $\hat{\theta}_j = t_{j-1} + (1-\theta) h$ for $\theta \in [0,1]$ 
and $j\in \{1,\dots,N\}$. 
Observe that in the deterministic case, where $\diff{W(t)}$ is replaced 
by $\diff{t}$, the second sum would disappear and we indeed recover the 
trapezoidal rule if $\theta = 0$. Further, the choice $\theta =
\frac{1}{2}$ yields the midpoint rule. 
In Section~\ref{sec:higherOrder} we also
show that the implementation of \eqref{eq:Trap} is straight-forward.

The remainder of this paper is organized as follows: In
Section~\ref{sec:str_error} we recall the definition of the fractional Sobolev
spaces $W^{\sigma,p}(0,T)$ and the associated Sobolev--Slobodeckij norm. In
addition, we fix some notation and collect a few martingale inequalities.
Section~\ref{sec:randQuad} and Section~\ref{sec:higherOrder} then contain the
error analysis of the quadrature rules \eqref{eq:Q} and \eqref{eq:Trap},
respectively. In Section~\ref{sec:num} we then present several numerical
experiments for the case of deterministic integrands with various degrees of
smoothness. In Section~\ref{sec:Poisson} we finally 
show that a Poisson process satisfies the conditions imposed on
the randomly shifted Riemann--Maruyama rule and state some numerical 
tests.


\section{Preliminaries} 
\label{sec:str_error}

First, let us recall the definition of fractional Sobolev spaces which are used
in order to determine the temporal regularity of the integrand. 
For $T \in (0,\infty)$ and $p \in [1,\infty)$ the \emph{Sobolev-Slobodeckij} 
norm of an integrable mapping $v \colon [0,T] \to \R$ is given by
\begin{align}
  \label{eq:normWsp1}
  \| v \|_{W^{\sigma,p}(0,T)} 
  = \Big( \int_0^T | v(t) |^p \diff{t} + \int_0^T \int_0^T \frac{|v(t) -
  v(s) |^p}{|t - s|^{1 + \sigma p}} \diff{t} \diff{s}
  \Big)^{\frac{1}{p}}
\end{align}
for $\sigma \in (0,1)$ and
\begin{align}
  \label{eq:normWsp2}
  \| v \|_{W^{\sigma,p}(0,T)} 
  = \Big( \int_0^T | v(t) |^p \diff{t} + \int_0^T | \dot{v}(t) |^p \diff{t}
  + \int_0^T \int_0^T \frac{|\dot{v}(t) - \dot{v}(s) |^p}{|t - s|^{1 + \sigma p}} 
  \diff{t} \diff{s}
  \Big)^{\frac{1}{p}}
\end{align}
for $\sigma \in (1,2)$.
We denote by $W^{\sigma,p}(0,T) \subset L^p(0,T)$ the subspace of all
$L^p$-integrable mappings $v \colon [0,T] \to \R$ such that
$\|v\|_{W^{\sigma,p}(0,T)} < \infty$. The space $W^{\sigma,p}(0,T)$ is called 
\emph{fractional Sobolev space}. It holds true that $W^{1,p}(0,T) \subset
W^{\sigma,p}(0,T)  \subset L^p(0,T)$ for all $\sigma \in (0,1)$. For further
details on fractional Sobolev spaces we refer the reader, for example, to 
\cite[Chapter~4]{demengel2012} or to the
survey papers \cite{dinezza2012} and \cite{simon1990}.

For the error analysis it is convenient to introduce a further probability space 
$(\Omega, \F, \P)$ which is of product form
\begin{align}
  \label{eq:Omega}
  (\Omega, \F, \P) := (\Omega_{W}\times\Omega_{\Theta}, \F^W\otimes 
  \F^{\Theta},
  \P_W \otimes \P_{\Theta}). 
\end{align}
Recall from Section~\ref{sec:intro} that $(\Omega_W,\F^W, (\F_t^W)_{t\in
[0,T]},\P_W)$ is the stochastic basis of the Wiener process $W$ and the
integrand $G$ in \eqref{eq:stochint}, while the family of random temporal grid
points $\pi_h^\Theta$ determined by the random
variable $\Theta$ is defined on $(\Omega_\Theta,\F^\Theta,\P_\Theta)$. In 
the following we denote by $\E_W[ \cdot ]$ and $\E_\Theta[ \cdot ]$ the
expectation with respect to the measures $\P_W$ and $\P_\Theta$, respectively.

For the error analysis with respect to the $L^p(\Omega)$-norm, $p \in
[2,\infty)$, we also require the following higher moment estimate of stochastic
integrals. For a proof we refer to \cite[Chapter 1, Theorem~7.1]{mao2008}. 

\begin{theorem} 
  \label{th:hoeld}
  Let $p \in [2,\infty)$ and $G \in L^p(\Omega_W; L^p(0,T))$ be stochastically
  integrable. Then, it holds true that
  \begin{align*}
    \E_W \Big[ \Big| \int_0^T G(t) \diff{W(t)} \Big|^p \Big] 
    &\leq \Big(\frac{p(p-1)}{2} \Big)^{\frac{p}{2}} T^{\frac{p-2}{2}} 
        \E_W \Big[ \int_0^T |G(t)|^p \diff{t}  \Big].
  \end{align*}
\end{theorem}

The error analysis also relies on a discrete time version
of the Burkholder--Davis--Gundy inequality. A proof is found in
\cite{burkholder1966}. 

\begin{theorem} \label{th:bdg}
  For each $p \in (1,\infty)$ there exist positive constants $c_p$ and $C_p$
  such that for every discrete time martingale $(X_n)_{n\in \N}$ and for every
  $n\in \N$ we have
  \begin{align*}
    c_p \Big \|  [X]_n^{\frac{1}{2}} \Big\|_{L^p(\Omega;\R^d)}
    \leq \Big \| \max_{i \in \{1,\dots,n\} } |X_i | \Big\|_{L^p(\Omega;\R^d)}
    \leq C_p \Big \| [X]_n^{\frac{1}{2}} \Big\|_{L^p(\Omega;\R^d)}
  \end{align*}
  where $[X]_n = \big| X_1 \big|^2 + \sum_{i=1}^{n-1} \big|X_{i+1} - X_i 
  \big|^2$ 
  denotes the quadratic variation of $(X_n)_{n\in \N}$ up to $n$.
\end{theorem}


\section{Error analysis of the lower order quadrature rule}
\label{sec:randQuad}

In this section we present the error analysis of the randomly shifted
Riemann--Maruyama quadrature rule defined in \eqref{eq:Q}. First, we state 
the assumptions on the integrand in the stochastic integral 
\eqref{eq:stochint}.

\begin{assumption}
  \label{as:G}
  The mapping $G \colon [0,T] \times \Omega_W \to \R$ is a $(\F_t^W)_{t \in 
  [0,T]}$-progressively measurable stochastic process such that there exist 
  $p \in [2,\infty)$ 
  and $\sigma \in (0,1)$ with
  \begin{align*}
    G \in L^p( \Omega_W ; W^{\sigma,p}(0,T) ).
  \end{align*}
  In addition, there exist $C_0 \in (0,\infty)$ and $h_0 \in (0,T]$
  with
  \begin{align}
    \label{eq:G_ini}
    \int_0^{h} \E_W\big[ |G(t)|^p \big] \diff{t} \le C_0 h^{\max(0,p\sigma - 
    \frac{p-2}{2}) } \quad \text{for all } h \leq h_0.
  \end{align}
\end{assumption}

Under Assumption~\ref{as:G} the stochastic process $G$ is stochastically
integrable and the It\^{o}-integral \eqref{eq:stochint} is well-defined. For 
more
details on stochastic integration we refer the reader, for instance, to
\cite[Chapter~17]{kallenberg2002} or \cite[Chapter~25]{klenke2014}.
Moreover, we stress that in the case $\sigma \in 
(0,\frac{1}{p})$ the stochastic process $G$ does not necessarily possess
continuous trajectories. In Section~\ref{sec:Poisson} we show that
a Poisson process satisfies all conditions of Assumption~\ref{as:G} for all
$p \in [2,\infty)$ and $\sigma \in (0,1)$ with $\sigma p < 1$.

\begin{remark}
  The condition \eqref{eq:G_ini} ensures that the $L^p(\Omega_W)$-norm
  of the process $G$ is not too explosive at $t = 0$. In Section~\ref{sec:num}
  we will show that Assumption~\ref{as:G} includes weak singularities of the form
  $[0,T] \ni t \mapsto t^{- \gamma}$ for $\gamma \in (0,\frac{1}{2})$. On the 
  other hand, if the integrand enjoys more regularity at $t=0$ but is nonzero,
  then one might apply the quadrature rule \eqref{eq:Q}
  to the integrand $\tilde{G}(t) := G(t) - G(0)$ to verify \eqref{eq:G_ini} 
  for larger values of $\sigma$.
\end{remark}

\begin{remark}
  The randomly shifted quadrature rule $Q_N^{\SRM}[G,\Theta]$ only 
  evaluates $G$
  on the randomized time points in $\pi_h(\Theta)$ determined by $\Theta 
  \sim \mathcal{U}(0,1)$. Because of this, the quadrature rule
  is independent of the choice of the representation of the equivalence class
  $G \in L^p(\Omega; W^{\sigma,p}(0,T))$ in the following sense:
  For all $\omega \in \Omega_W$ with $G(\cdot, \omega) \in W^{\sigma,p}(0,T)$
  let $G_i(\cdot, \omega)$, $i \in \{1,2\}$, be two representations of the same
  equivalence class in $W^{\sigma, p}(0,T)$. 
  Then it follows from 
  \begin{align*}
    G_1(t,\omega) = G_2(t,\omega)
  \end{align*}
  for almost all $t \in [0,T]$ that
  \begin{align*}
    G_1(\Theta_j,\omega ) = G_2(\Theta_j, \omega) \quad  \text{
    $\P_\Theta$-almost surely in } \Omega_\Theta
  \end{align*}
  for every $j \in \{1,\ldots,N\}$, and hence $G_1(\Theta_j) = G_2(\Theta_j)$  
  $\P$-almost surely on $\Omega = \Omega_W \times \Omega_\Theta$.
\end{remark}

First, let us prove a lemma, where we insert an arbitrary but fixed value 
$\theta \in [0,1]$ into \eqref{eq:Q} instead of the random variable $\Theta$.

\begin{lemma}
  \label{lem:main2}
  Let Assumption~\ref{as:G} be satisfied with $p \in [2,\infty)$, $\sigma \in
  (0,1)$, $C_0 \in (0,\infty)$, and $h_0 \in (0,T]$. Further, let $\theta \in
  [0,1]$ be arbitrary and $\theta_j = t_{j-1} + \theta h$ for $ j\in 
  \{1,\dots,N\}$ with $\theta_0 = 0$ and $\theta_{N+1} = T$. Then, there 
  exists $C(p) \in (0,\infty)$ depending only on $p \in [2,\infty)$ with 
  \begin{align*}
    &\big\| I[G] - Q^\SRM_N[G,\theta] \big\|_{L^p(\Omega_W)}\\ 
    &\quad \leq  C(p) 
    h^{\frac{p-2}{2p}} \Big(\int_{0}^{\theta_1} 
    \E_W \big[ |G(t)|^p \big] \diff{t} \Big)^{\frac{1}{p}}\\
    &\qquad + C(p)
    h^{\frac{p-2}{2p}} \Big( \sum_{j = 1}^N
    \Big(\int_{\theta_j}^{\theta_{j+1}} 
    \E_W \big[ |G(t)-  G(\theta_j)|^p \big]
    \diff{t} \Big)^{\frac{2}{p}} \Big)^{\frac{1}{2}}
  \end{align*}
  for all $N \in \N$ with $\frac{T}{N} = h \le h_0$ and almost every $\theta
  \in [0,1]$.
\end{lemma}

\begin{proof}
  Analogously to \eqref{eq:dist}, we have for all
  $j \in \{0, 1,\ldots,N\}$ and every $\theta \in [0,1]$ that
  \begin{align*}
    | \theta_{j+1} - \theta_j | \le h
  \end{align*}
  by definition of $(\theta_j)_{j \in \{0,\dots,N+1 \}}$. 
  We abbreviate the time discrete error term by
  \begin{align*}
    E^n_{\theta} 
    &= \int_{0}^{\theta_n } G(t) \diff{W(t)} 
    - \sum_{j = 1}^{n-1} G(\theta_j)\big( W(\theta_{j+1}) - W(\theta_{j}) \big)\\
    &= \int_{0}^{\theta_1 } G(t) \diff{W(t)} + \sum_{j = 1}^{n-1} 
    \int_{\theta_j}^{\theta_{j+1}} \big( G(t) -  G(\theta_j) \big) \diff{W(t)}
  \end{align*}
  for $n \in \{1,\dots,N+1 \}$. Then, we can write the error of the quadrature
  rule \eqref{eq:Q} as
  \begin{align*}
    \big\| I[G] - Q^\SRM_N[G,\theta] \big\|_{L^p(\Omega_W)}^p  
    = \E_W \big[  \big| E^{N+1}_{\theta} \big|^p \big].
  \end{align*}
  Furthermore, it follows from Assumption~\ref{as:G} and 
  Theorem~\ref{th:hoeld}
  that $E^n_{\theta} \colon \Omega_W \to \R$ is an element of
  $L^p(\Omega_W)$ for every $n \in \{1,\ldots,N+1\}$.  
  In addition, $E^n_{\theta}$ 
  is measurable with respect to the $\sigma$-algebra 
  $\F^W_{\theta_n}$.  
  Since we obtain for all $1 \leq m \leq n \leq N+1$ that
  \begin{align*}
    \E_W\Big[ E^n_{\theta} \Big| \F^W_{\theta_m} \Big]
    &= \E_W\Big[ \int_{0}^{\theta_1 } G(t) \diff{W(t)} 
    + \sum_{j = 1}^{n-1}  \int_{\theta_j}^{\theta_{j+1}} 
    \big( G(t) -  G(\theta_j) \big) 
    \diff{W(t)} \Big| \F^W_{\theta_m} \Big]\\
    &=  \int_{0}^{\theta_1 } G(t) \diff{W(t)} 
    + \sum_{j = 1}^{m-1}  \int_{\theta_j}^{\theta_{j+1}} 
    \big( G(t) -  G(\theta_j) \big) \diff{W(t)} \\
    &\quad + \E_W\Big[ \sum_{j = m}^{n-1}  
    \int_{\theta_j}^{\theta_{j+1}} 
    \big( G(t) -  G(\theta_j) \big) 
    \diff{W(t)} \Big| \F^W_{\theta_m} \Big]
    = E^m_{\theta},
  \end{align*}
  the process $(E^n_{\theta})_
  {n \in \{1,\dots,N+1\}}$ is a discrete time martingale with respect to the 
  filtration $\big(\F^W_{\theta_n}\big)_{n\in \{1,\dots,N+1\}}$.
  From an application of the Burkholder--Davis--Gundy inequality from
  Theorem~\ref{th:bdg} and the triangle inequality we obtain 
  \begin{align*}
    &\Big(  \E_W \Big[  
    \big| E^{N+1}_{\theta}  \big|^p \Big] \Big)^{\frac{1}{p}} \\
    &\quad \leq C_p \Big( \E_W \Big[  \Big( \big| E^1_{\theta} \big |^2 
    + \sum_{j=1}^{N} \big| E^{j+1}_{\theta} - E^j_{\theta} \big|^2  
    \Big)^\frac{p}{2} \Big] \Big)^{\frac{1}{p}}\\
    &\quad = C_p \Big( \Big\| \big| E^1_{\theta} \big |^2 
    + \sum_{j=1}^{N} \big| E^{j+1}_{\theta} - E^j_{\theta} \big|^2  
    \Big\|_{L^{\frac{p}{2}}(\Omega_W)} \Big)^{\frac{1}{2}}\\
    &\quad \leq C_p \Big(  \big\| E^1_{\theta} \big\|^2_{L^p(\Omega_W)}
    + \sum_{j = 1}^N \big\| E^{j+1}_{\theta} - E^j_{\theta}
    \big\|^2_{L^p(\Omega_W)} \Big)^\frac{1}{2} \\
    &\quad \leq C_p \big\| E^1_{\theta} \big\|_{L^p(\Omega_W)}
    + C_p \Big( \sum_{j = 1}^N \big\| E^{j+1}_{\theta} - E^j_{\theta}
    \big\|^2_{L^p(\Omega_W)} \Big)^\frac{1}{2} 
    =: C_p \big( X_1  + X_2 \big),
  \end{align*}
  where we will consider $X_1$ and $X_2$ separately in the following. By 
  making
  use of Theorem~\ref{th:hoeld} we obtain the estimate for $X_1$ 
  \begin{align*}
    X_1^{p}
    = \Big\| \int_{0}^{\theta_1} G(t) \diff{W(t)} 
    \Big\|_{L^p(\Omega_W)}^p 
    \leq \Big(\frac{p(p-1)}{2} \Big)^{\frac{p}{2}} h^{\frac{p-2}{2}} 
    \int_{0}^{\theta_1} 
    \E_W \big[ |G(t)|^p \big] \diff{t},
  \end{align*}
  since $\theta_1 \le h$.
  To estimate $X_2$ we again apply Theorem~\ref{th:hoeld} and obtain that  
  \begin{align*}
    X_2^2
    &= \sum_{j = 1}^N 
    \big\| E^{j+1}_{\theta} - E^j_{\theta}
    \big\|^2_{L^p(\Omega_W)}\\
    &= \sum_{j = 1}^N 
    \Big\| \int_{\theta_j}^{\theta_{j+1}} 
    \big( G(t) -  G(\theta_j) \big)
    \diff{W(t)} \Big\|^2_{L^p(\Omega_W)}\\
    &\le \frac{p(p-1)}{2} h^{\frac{p-2}{p}} 
    \sum_{j = 1}^N \Big(\int_{\theta_j}^{\theta_{j+1}}
    \E_W \big[ |G(t)-  G(\theta_j)|^p \big]
    \diff{t} \Big)^{\frac{2}{p}}.
  \end{align*}
  Altogether, this yields the assertion with
  $C(p) = C_p  (\frac{p(p-1)}{2} )^{\frac{1}{2}}$.
\end{proof}

\begin{lemma}
  \label{lem:measureable}
  Let Assumption~\ref{as:G} be satisfied with $p \in [2,\infty)$, $\sigma \in
  (0,1)$, $C_0 \in (0,\infty)$, and $h_0 \in (0,T]$. 
  For every $h = \frac{T}{N} \le h_0$, $N \in \N$, consider for $n \in
  \{1,\ldots,N\}$ and $\theta \in [0,1]$ the discrete time error process 
  \begin{align}
    \label{eq:E_n_theta}
    E^n_{\theta} 
    &= \int_{0}^{\theta_1 } G(t) \diff{W(t)} + \sum_{j = 1}^{n-1} 
    \int_{\theta_j}^{\theta_{j+1}} \big( G(t) -  G(\theta_j) \big) \diff{W(t)},
  \end{align}
  where $\theta_j = (j - 1 + \theta) h$, $j \in \{1,\ldots,N\}$. 
  Then the mapping 
  \begin{align*}
    [0,1] \times \Omega_W  \ni (\theta, \omega_W) \mapsto 
    E_\theta^n(\omega_W) \in \R
  \end{align*}
  is $\B(0,1) \otimes \F_{t_n}^W / \B(\R)$-measurable. 
\end{lemma}

\begin{proof}
  Recall that for every stochastically integrable process $G \colon [0,T]
  \times \Omega_W \to \R$ the stochastic It\^o-integral
  \begin{align*}
    \int_0^t G(s) \diff{W(s)}
  \end{align*}
  considered as a stochastic process with respect to its upper integration
  limit $t \in [0,T]$ is $(\F_t^W)_{t \in [0,T]}$-progressively measureable.
  From this it follows that the mapping
  \begin{align*}
    [0,1] \times \Omega_W \ni (\theta, \omega_W) 
    \mapsto E_\theta^1(\omega_W) =  \Big( \int_0^{h \theta} G(s) \diff{W(s)} 
    \Big) (\omega_W)
  \end{align*}
  is $\B(0,1) \otimes \F_{t_1}^W / \B(\R)$-measureable. 
  
  For the same reasons, due to $\theta_j \le t_n$ for all $j \in
  \{0,\ldots,n\}$, and since $G$ is assumed to be $(\F_t^W)_{t \in
  [0,T]}$-progressively measureable we also obtain the claimed
  product measurability of all other summands in \eqref{eq:E_n_theta}.  
\end{proof}

We now state and prove the error estimate of the randomly shifted
Riemann--Maruyama quadrature rule defined in \eqref{eq:Q}.

\begin{theorem}
  \label{th:main2}
  Let Assumption~\ref{as:G} be satisfied with $p \in [2,\infty)$, $\sigma \in
  (0,1)$, $C_0 \in (0,\infty)$, and $h_0 \in (0,T]$ and let $\Theta \colon 
  \Omega \to [0,1]$ be a uniformly distributed random variable which is 
  independent of the stochastic processes $G$ and $W$. Then, there exists 
  $C(p) \in (0,\infty)$ depending only on $p \in [2,\infty)$ with 
  \begin{align*}
    &\big\| I[G] - Q^\SRM_N[G,\Theta] \big\|_{L^p(\Omega)} \\
    &\qquad \le C(p) \big( C_0^{\frac{1}{p}}
    h_0^{\max(0,\frac{p-2}{2p} - \sigma) } +
    T^{\frac{p-2}{2p}} \|G\|_{L^p(\Omega_W;W^{\sigma,p}(0,T))} \big) 
    h^{\sigma}
  \end{align*}
  for all $N \in \N$ with $\frac{T}{N} = h \le h_0$.
\end{theorem}

\begin{proof}
  As in Lemma~\ref{lem:measureable} 
  we abbreviate the time discrete error process $E_\theta^n$, $n \in
  \{1,\ldots,N+1\}$, for each value of $\theta \in [0,1]$ by
  \begin{align*}
    E^n_{\theta} 
    &= \int_{0}^{\theta_1 } G(t) \diff{W(t)} + \sum_{j = 1}^{n-1} 
    \int_{\theta_j}^{\theta_{j+1}} \big( G(t) -  G(\theta_j) \big) \diff{W(t)},
  \end{align*}
  where $\theta_j = (j - 1 + \theta) h$.

  By $E_{\Theta}^n$ we then denote the composition of the mappings $\Omega \ni
  (\omega_W, \omega_\Theta) \mapsto ( \Theta(\omega_\Theta), \omega_W) \in
  (0,1) \times \Omega_W$ and $(0,1) \times \Omega_W \ni (\theta, \omega_W)
  \mapsto E_\theta^n(\omega_W) \in \R$. Clearly, the random variable
  $E_\Theta^n$ is then $\F^W_{t_n} \otimes \F^\Theta / \B(\R)$-product
  measureable for all $n \in \{1,\ldots,N+1\}$. 

  Next, we give an estimate of the $L^p(\Omega)$-norm of the
  error of the quadrature rule \eqref{eq:Q} 
  \begin{align*}
    \big\| I[G] - Q^\SRM_N[G,\Theta] \big\|_{L^p(\Omega)}^p  
    = \E_\Theta \big[ \E_W \big[  \big| E^{N+1}_{\Theta} \big|^p \big] \big].
  \end{align*}
  Using Lemma~\ref{lem:main2}, we now obtain that for almost every 
  $\omega_{\Theta} \in \Omega_{\Theta}$
  \begin{align*}
    &\Big(  \E_W \Big[  
    \big| E^{N+1}_{\Theta(\omega_{\Theta})}  \big|^p \Big] \Big)^{\frac{1}{p}} \\
    &\quad \leq  C(p) 
    h^{\frac{p-2}{2p}} \Big(\int_{0}^{\Theta_1(\omega_{\Theta})} 
    \E_W \big[ |G(t)|^p \big] \diff{t} \Big)^{\frac{1}{p}}\\
    &\qquad + C(p)
    h^{\frac{p-2}{2p}} \Big( \sum_{i = 1}^N
    \Big(\int_{\Theta_i(\omega_{\Theta})}^{\Theta_{i+1}(\omega_{\Theta})} 
    \E_W \big[ |G(t)-  G(\Theta_i(\omega_{\Theta}))|^p \big]
    \diff{t} \Big)^{\frac{2}{p}} \Big)^{\frac{1}{2}},
  \end{align*}
  where $C(p) =   C_p  (\frac{p(p-1)}{2} )^{\frac{1}{2}}$
  and $\Theta_j = (j - 1 + \Theta) h$. 
  Hence, after applying the norm $( \E_\Theta [ ( \cdot)^p ] )^{\frac{1}{p}}$ we
  get
  \begin{align}
    \label{eq3:term1}
    \begin{split}
      \big\| E^{N+1}_{\Theta}  \big\|_{L^p(\Omega)}
      &= \big(  \E_{\Theta} \big[ \E_W \big[  
      | E^{N+1}_{\Theta}  |^p \big] \big] \big)^{\frac{1}{p}} \\
      &\leq  C(p) h^{\frac{p-2}{2p}} \Bigg[ 
      \Big( \E_\Theta \Big[ \int_{0}^{\Theta_1}
      \E_W\big[ |G(t)|^p \big] \diff{t}  
      \Big] \Big)^{\frac{1}{p}} \\
      &\quad + \Big( \E_\Theta \Big[ \Big(  \sum_{i=1}^{N} \Big(
      \int_{\Theta_i}^{\Theta_{i+1}} \E_W \big[ 
      \big| G(t) -  G(\Theta_i) \big|^p \big] \diff{t} \Big)^{\frac{2}{p}}
      \Big)^{\frac{p}{2}}  \Big] \Big)^{\frac{1}{p}} \Bigg].
    \end{split}
  \end{align}
  Due to $h \le h_0$ we have by condition \eqref{eq:G_ini} for the first
  term that
  \begin{align*}
    \E_\Theta \Big[ \int_{0}^{\Theta_{1}} \E_W \big[
    \big| G(t) \big|^p \big] \diff{t} \Big]
    &= \frac{1}{h} \int_{0}^h \int_{0}^{\theta} 
    \E_W \big[ | G(t) |^p \big] \diff{t} \diff{\theta}\\
    &\le \int_0^h \E_W \big[ | G(t) |^p \big] \diff{t} 
    \le C_0 h^{\max(0,p\sigma - 
      \frac{p-2}{2}) }.
  \end{align*}
  Since $|t - \Theta_i| \le | \Theta_{i+1} - \Theta_i| \le h$ is fulfilled in
  the second summand on the right hand side of \eqref{eq3:term1}
  we further estimate the second sum by
  \begin{align*}
    & \E_\Theta \Big[ \Big(  \sum_{i=1}^{N} \Big(
    \int_{\Theta_i}^{\Theta_{i+1}} \E_W \big[ 
    \big| G(t) -  G(\Theta_i) \big|^p \big] \diff{t} \Big)^{\frac{2}{p}}
    \Big)^{\frac{p}{2}}  \Big] \\
    &\quad \le N^{\frac{p - 2}{2}} \sum_{i = 1}^N
    \E_\Theta \Big[ \int_{\Theta_i}^{\Theta_{i+1}} \E_W \big[ 
    \big| G(t) -  G(\Theta_i) \big|^p \big] \diff{t} \Big] \\
    &\quad \le N^{\frac{p-2}{2}} h^{1 + p \sigma} 
    \sum_{i = 1}^N \E_\Theta \Big[ \int_{\Theta_{i}}^{\Theta_{i+1}} 
    \frac{\E_W \big[ \big| G(t) - G(\Theta_i)
    \big|^p \big] }{| t - \Theta_i |^{1+p\sigma}} \diff{t} \Big] \\
    &\quad \le N^{\frac{p-2}{2}} h^{1+p\sigma} 
    \sum_{i = 1}^N \int_0^T \E_\Theta \Big[
    \frac{\E_W \big[ \big| G(t) - G(\Theta_i)
    \big|^p \big] }{| t - \Theta_i |^{1+p\sigma}} \Big] \diff{t} \\
    &\quad = N^{\frac{p-2}{2}} h^{1+p\sigma} \sum_{i = 1}^N \int_0^T \frac{1}{h}
    \int_{t_{i-1}}^{t_i} \frac{\E_W \big[ \big| G(t) - G(s)
    \big|^p \big] }{| t - s |^{1+p\sigma}} \diff{s} \diff{t}\\
    &\quad \leq N^{\frac{p-2}{2}} h^{p\sigma} \| G
    \|_{L^p(\Omega_W;W^{\sigma,p}(0,T))}^p, 
  \end{align*}
  where we made use of the fact that $\Theta_i \sim \mathcal{U}(t_{i-1},t_i)$ 
  in
  the second last step.
  The assertion then follows at once after inserting the last two estimates into
  \eqref{eq3:term1} and by
  noting that $N^{\frac{p-2}{2p}} h^{\frac{p-2}{2p}} =
  T^{\frac{p-2}{2p}}$ and $\max(0,\sigma - \frac{p-2}{2p}) +
  \frac{p-2}{2p} = \max(\frac{p-2}{2p}, \sigma) \ge \sigma$.
\end{proof}

\begin{remark}
  Let us briefly compare the error estimate of Theorem~\ref{th:main2} to the
  standard case with H\"older regularity, where it is assumed that
  $G \in C^\gamma([0,T];L^p(\Omega_W))$, $\gamma \in (0,1)$. In this case 
  the
  random shift of the mesh $\pi_h$ is not required and 
  the standard Riemann--Maruyama quadrature rule \eqref{eq:RM} converges with
  order $\gamma$.  
  
  Since every function in $C^\gamma([0,T];L^p(\Omega_W)) \cap 
  L^p((0,T) \times \Omega_W)$ is also an element of $L^p(\Omega_W; 
  W^{\sigma,p}(0,T) )$ for all $\sigma \in (0,\gamma)$ the error estimate in
  Theorem~\ref{th:main2} guarantees that $\gamma$ is essentially also a lower
  bound for the order of convergence of the quadrature rule \eqref{eq:Q}. 
  However, as we will also see in Section~\ref{sec:num}, one readily finds
  integrands $G \in C^\gamma([0,T];L^p(\Omega_W)) \cap L^p(\Omega_W;
  W^{\sigma,p}(0,T) )$ with $\sigma > \gamma$. For example, the process
  $G(t) := t^{\frac{1}{4}} + W(t)$, $t \in [0,T]$, is an element of
  $C^{\gamma}([0,T];L^2(\Omega_W))$ with $\gamma = \frac{1}{4}$. 
  However, 
  it is simple to verify that we also have $G \in L^2(\Omega_W;
  W^{\sigma,2}(0,T) )$ for every $\sigma \in (0,\frac{1}{2})$.
\end{remark}


\section{Higher order quadrature rule}
\label{sec:higherOrder}

In this section we present the details on the higher order quadrature rule 
\eqref{eq:Trap}.  To the best of our knowledge there 
is little literature on higher order quadrature rules for It\^{o}-integrals.
When estimating the solution of a stochastic differential equation with higher
order Runge--Kutta schemes, our quadrature rule with $\theta = 0$ appears 
as a by-product. See, for example, in \cite[Chapter 12]{kloeden1999} and 
\cite{Roessler2005} with classical and stricter regularity assumptions on the 
integrand. 
For further results on higher order Runge--Kutta schemes we also refer the 
reader to \cite[Chapter 1]{milstein2004}, where schemes containing a 
derivative of $g$ are considered.
Let us mention that the quadrature rule \eqref{eq:Trap} can also be seen as 
a derivative-free version of the Wagner--Platen scheme, see 
\cite{kloeden1999}. This has been studied in \cite{przybylowicz2009} under 
classical smoothness assumptions, that is, $g \in C^1([0,T])$ with a globally 
Lipschitz continuous derivative.  
For the case of arbitrary $\theta \in [0,1]$ as, for example, the midpoint rule 
when choosing $\theta = \frac{1}{2}$, there are no known results to us. 
Furthermore, the regularity assumption is the standard literature is stricter 
than in our work.

First we state the conditions for our error analysis. 

\begin{assumption}
  \label{as:g}
  There exist $p \in [2,\infty)$ and $\sigma \in (0,1)$ such that
  the mapping $g \colon [0,T] \to \R$ is an element of $W^{1+ \sigma, p}(0,T)$. 
\end{assumption}

Let us take note that Assumption~\ref{as:g} and the
Sobolev embedding theorem ensure the existence of a continuous 
representative of the integrand. Hence, the point evaluation of $g$ on the 
deterministic grid points in \eqref{eq:Trap} is well-defined.
Because of this the artificial randomization of the freely 
selectable parameter value $\theta \in [0,1]$ is not necessary.

Still, different choices of $\theta$ can affect the error. 
While the rate of convergence does not change when varying $\theta$, it 
can have an effect on the error constant.
For each value of $\theta$ we then define the two points
\begin{align*}
  \theta_j = t_{j-1} + \theta h, 
  \quad \hat{\theta}_j = t_{j-1} + (1-\theta) h,
  \quad j\in \{1,\dots,N\}, 
\end{align*}
where as before $h = \frac{T}{N}$, $N \in \N$, and $t_j = jh$, $j \in
\{0,\ldots,N\}$. Also we denote the midpoint between two grid points
$t_{j-1}$ and $t_{j}$ by $t_{j -\frac{1}{2}}$, that is, 
\begin{align*}
  t_{j -\frac{1}{2}} = \frac{t_{j-1} + t_{j}}{2}, \quad j \in \{1,\ldots,N\}. 
\end{align*} 
Then, the quadrature rule studied in this section is given by
\begin{align*}
  \begin{split}
    Q^{\mathrm{Trap}}_N [g] 
		&= \sum_{j=1}^{N} \frac{1}{2} (g(\theta_j) + g(\hat{\theta}_j)) (W(t_j)- 
		W(t_{j-1}))\\
		&\quad + \sum_{j=1}^{N} \frac{1}{h} (g(t_j) - g(t_{j-1}) ) 
		\int_{t_{j-1}}^{t_j} (t - t_{j -\frac{1}{2}} ) \diff{W(t)}.
  \end{split}
\end{align*}
Let us observe that the parameter value $\theta = 0$ yields the stochastic
trapezoidal rule. This choice of $\theta$ also admits the
practical advantage that it only requires $N+1$ function evaluations of the
integrand $g$, since then $\theta_j = t_{j-1}$ and $\hat{\theta}_j = t_j$. 
Furthermore, choosing $\theta = 0.5$ we obtain the stochastic midpoint 
rule. Therefore, our general approach offers an analysis that covers two well 
known rules at once.

\begin{theorem}
  \label{th:main3}
  Let Assumption~\ref{as:g} be satisfied with $p \in [2,\infty)$ and $\sigma
  \in (0,1)$. Then, for all $N \in \N$ with $\frac{T}{N} = h$ 
  it holds true that 
  \begin{align*}
    \big\| I[g] - Q^{\mathrm{Trap}}_{N}[g] \big\|_{L^p(\Omega)} 
    \le C_p \big( 2p(p-1) \big)^{\frac{1}{2}} T^{\frac{p-2}{2p}}
    h^{1 + \sigma} \| g \|_{W^{1+\sigma,p}(0,T)}.
  \end{align*}
\end{theorem}

The proof of Theorem~\ref{th:main3} relies on the following lemma, which
contains a useful representation of the error of the quadrature formula
\eqref{eq:Trap}.

\begin{lemma} 
  \label{lem:error}
  Let Assumption~\ref{as:g} be satisfied with $p \in [2, \infty)$, $\sigma \in
  (0,1)$. Then, for every $N \in \N$ the discrete time error process $(E^n)_{n
  \in \{0,\ldots,N\}}$ of the quadrature rule \eqref{eq:Trap} defined by $E^0
  := 0$ and  
  \begin{align*}
    E^n 
		&= \sum_{j=1}^{n} 
		\int_{t_{j-1}}^{t_j} \Big( g(t) - \frac{1}{2} (g(\theta_j) + g(\hat{\theta}_j))
	  - \frac{1}{h} (g(t_j) - g(t_{j-1}) ) (t - t_{j -\frac{1}{2}}) \Big)
    \diff{W(t)} 
  \end{align*}
  for  $n \in \{1,\ldots,N\}$, is a discrete time $(\F_{t_n})_{n \in
  \{0,\ldots,N\}}$-adapted $L^p(\Omega_W)$-martingale. 
  Moreover, it holds true that 
  \begin{align}
    \label{eq4:rep_error}
    \begin{split}
		  E^n 
		  &= \frac{1}{h} \sum_{j=1}^{n} \int_{t_{j-1}}^{t_j} \int_{t_{j-1}}^{t_j}
      \Big( \int_{t_{j -\frac{1}{2}}}^{t} 	\big( \dot{g}(s) - \dot{g}(r) \big)
      \diff{s}\\ 
		  &\qquad -\frac{1}{2}  
		  \int_{t_{j -\frac{1}{2}}}^{\theta_j} \big( \dot{g}(s) - \dot{g}(r) \big)
      \diff{s} -\frac{1}{2} 
		  \int_{t_{j -\frac{1}{2}}}^{\hat{\theta}_j} \big( \dot{g}(s) - \dot{g}(r)
      \big) \diff{s} \Big) 
		  \diff{r} \diff{W(t)}	
    \end{split}
	\end{align}
  for all $n \in \{1,\dots,N \}$.
\end{lemma}

\begin{proof}
  The martingale property and the $L^p(\Omega_W)$-integrability follow directly
  from the definition of $E^n$ and the fact that $g \in W^{1+\sigma,p}(0,T)$
  implies the boundedness of $g$. In order to prove \eqref{eq4:rep_error} 
  let us rewrite $g(\theta_j) + g(\hat{\theta}_j)$ in a suitable way by
	\begin{align*}
		g(\theta_j) + g(\hat{\theta}_j)
		= 2 g(t_{j -\frac{1}{2}}) + \int_{t_{j -\frac{1}{2}}}^{\theta_j} \dot{g}(s)
    \diff{s}  
		+ \int_{t_{j -\frac{1}{2}}}^{\hat{\theta}_j} \dot{g}(s) \diff{s},
	\end{align*}
  where $\dot{g}$ denotes the weak derivative of $g \in W^{1+\sigma,p}(0,T)$. 
  Therefore, we have for all $t \in [t_{j-1},t_j]$ that
  \begin{align*}
    g(t) - \frac{1}{2} (g(\theta_j) + g(\hat{\theta}_j))
    &= g(t) - g(t_{j -\frac{1}{2}}) - \frac{1}{2} \int_{t_{j
    -\frac{1}{2}}}^{\theta_j} \dot{g}(s) \diff{s}  
    - \frac{1}{2} \int_{t_{j -\frac{1}{2}}}^{\hat{\theta}_j} \dot{g}(s) \diff{s}.
  \end{align*}
  Inserting this into the definition of $E^n$ then yields the three terms
  \begin{align*}
    E^n = \sum_{j = 1}^n \big( X_a^j - \frac{1}{2} X_b^j - X_c^j \big),
  \end{align*}
  where
  \begin{align*}
    X_a^j &= \int_{t_{j-1}}^{t_j} \big( g(t) - g(t_{j -\frac{1}{2}}) \big)
    \diff{W(t)},\\ 
    X_b^j &= \int_{t_{j-1}}^{t_j} \Big( \int_{t_{j -\frac{1}{2}}}^{\theta_j} \dot{g}(s) 
    \diff{s} 
    + \int_{t_{j -\frac{1}{2}}}^{\hat{\theta}_j} \dot{g}(s) \diff{s}
    \Big) \diff{W(t)},\\
    X_c^j &= \frac{1}{h} (g(t_j) - g(t_{j-1}) ) \int_{t_{j-1}}^{t_j} (t - t_{j
    -\frac{1}{2}}) \diff{W(t)}. 
  \end{align*}
  In the following let $j \in \{1, \dots, n \}$ be arbitrary. 
  For the term $X_{c}^j$ we then obtain 
	\begin{align*}
		X_c^j 
		&= \frac{1}{h} (g(t_j) - g(t_{j-1}) ) 
    \int_{t_{j-1}}^{t_j} (t - t_{j -\frac{1}{2}}) \diff{W(t)} \\
		&= \frac{1}{h} \int_{t_{j-1}}^{t_j} \dot{g}(r) \diff{r}
		\int_{t_{j-1}}^{t_j} \int^{t}_{t_{j -\frac{1}{2}}} \diff{s} \diff{W(t)} \\
		&= \frac{1}{h} \int_{t_{j-1}}^{t_j} \int^{t}_{t_{j -\frac{1}{2}}}
    \int_{t_{j-1}}^{t_j}  \dot{g}(r) \diff{r} \diff{s} \diff{W(t)} . 
	\end{align*}
	This now enables us to write
	\begin{align*}
		X_{a}^j - X_{c}^j
		& = \int_{t_{j-1}}^{t_j} \int_{t_{j -\frac{1}{2}}}^{t} \dot{g}(s) \diff{s}
    \diff{W(t)} 
		- \frac{1}{h} \int_{t_{j-1}}^{t_j} \int^{t}_{t_{j -\frac{1}{2}}}
    \int_{t_{j-1}}^{t_j}  
		\dot{g}(r) \diff{r} \diff{s} \diff{W(t)} \\
		& = \frac{1}{h} \int_{t_{j-1}}^{t_j} \int_{t_{j-1}}^{t_j} \int_{t_{j
    -\frac{1}{2}}}^{t}  
		\big( \dot{g}(s) - \dot{g}(r) \big) \diff{s} \diff{r} \diff{W(t)}.
	\end{align*}
	Further, due to the identity
	$\theta_j - t_{j -\frac{1}{2}} = - (\hat{\theta}_j - t_{j -\frac{1}{2}})$
  we have for the term $X_b^j$ that
	\begin{align*}
		X_b^j
    &= \int_{t_{j-1}}^{t_j} \int_{t_{j -\frac{1}{2}}}^{\theta_j} \dot{g}(s)
    \diff{s} \diff{W(t)}
		+ \int_{t_{j-1}}^{t_j} \int_{t_{j -\frac{1}{2}}}^{\hat{\theta}_j} \dot{g}(s)
    \diff{s} \diff{W(t)}\\
		&= \frac{1}{h}\int_{t_{j-1}}^{t_j} \int_{t_{j-1}}^{t_j} \int_{t_{j
    -\frac{1}{2}}}^{\theta_j}  \dot{g}(s) \diff{s} \diff{r}	\diff{W(t)}
		+ \frac{1}{h}\int_{t_{j-1}}^{t_j} \int_{t_{j-1}}^{t_j} 
		\int_{t_{j -\frac{1}{2}}}^{\hat{\theta}_j} 
		\dot{g}(s) \diff{s} \diff{r} \diff{W(t)}\\
		&\quad - \frac{\theta_j - t_{j -\frac{1}{2}}}{h} \int_{t_{j-1}}^{t_j}
    \int_{t_{j-1}}^{t_j}  \dot{g}(r) \diff{r} \diff{W(t)}
    - \frac{\hat{\theta}_j - t_{j -\frac{1}{2}}}{h} 
	  \int_{t_{j-1}}^{t_j} \int_{t_{j-1}}^{t_j}  \dot{g}(r) \diff{r}
    \diff{W(t)}\\ 
		&= \frac{1}{h} \int_{t_{j-1}}^{t_j} \int_{t_{j-1}}^{t_j} 
    \Big( \int_{t_{j -\frac{1}{2}}}^{\theta_j} \big( \dot{g}(s) - \dot{g}(r) \big)
    \diff{s} + 
    \int_{t_{j -\frac{1}{2}}}^{\hat{\theta}_j} \big( \dot{g}(s) - \dot{g}(r) \big)
    \diff{s} \Big) \diff{r} \diff{W(t)}.
	\end{align*}
	Altogether, this completes the proof of \eqref{eq4:rep_error}.
  \end{proof}

This lemma in mind, we now present our proof of the main result of this 
section.

\begin{proof}[Proof of Theorem~\ref{th:main3}]
  Let $N \in \N$ be arbitrary. Due to Lemma~\ref{lem:error} we know that the
  discrete time error process $(E^n)_{n \in \{0,\ldots,N\}}$ is a $p$-fold
  integrable martingale with respect to the filtration $(\F^W_{t_n})_{n \in
  \{0,\ldots,N\}}$. Thus, an application of Theorem~\ref{th:bdg} yields
  \begin{align*}
    \big\| \max_{n \in \{0,\ldots,N\}} | E^n| \big \|_{L^p(\Omega_W)}
    &\le C_p \Big\| \Big( \big | E^0 \big|^2 
		+  \sum_{j=0}^{N-1} \big| E^{j+1} - E^j \big|^2  
		\Big)^{\frac{1}{2}} \Big \|_{L^p(\Omega_W)}.
  \end{align*}
  After inserting $E^0 = 0$ and the representation \eqref{eq4:rep_error} we
  obtain by an application of the triangle inequality 
  \begin{align} 
    \label{4eq:errorSplit}
    \begin{split}
      &\big\| \max_{n \in \{1,\ldots,N\}} | E^n| \big \|_{L^p(\Omega_W)}\\
      &\quad \le C_p \frac{1}{h} \Big\| \Big( \sum_{j=1}^{N} \Big| 
      \int_{t_{j-1}}^{t_j} \int_{t_{j-1}}^{t_j} 
	    \int_{t_{j -\frac{1}{2}}}^{t} \big( \dot{g}(s) - \dot{g}(r) \big)
      \diff{s} \diff{r} \diff{W(t)}
      \Big|^2 \Big)^{\frac{1}{2}} \Big \|_{L^p(\Omega_W)} \\
      &\qquad + C_p \frac{1}{2h} \Big\| \Big( \sum_{j=1}^{N} \Big| 
      \int_{t_{j-1}}^{t_j} \int_{t_{j-1}}^{t_j} 
      \int_{t_{j -\frac{1}{2}}}^{\theta_j} \big( \dot{g}(s) - \dot{g}(r) \big)
      \diff{s} \diff{r} \diff{W(t)}
      \Big|^2 \Big)^{\frac{1}{2}} \Big \|_{L^p(\Omega_W)} \\
      &\qquad + C_p \frac{1}{2h} \Big\| \Big( \sum_{j=1}^{N} \Big| 
      \int_{t_{j-1}}^{t_j} \int_{t_{j-1}}^{t_j} 
      \int_{t_{j -\frac{1}{2}}}^{\hat{\theta}_j} \big( \dot{g}(s) - \dot{g}(r)
      \big) \diff{s} \diff{r} \diff{W(t)}
      \Big|^2 \Big)^{\frac{1}{2}} \Big \|_{L^p(\Omega_W)}.
    \end{split}
  \end{align}
  All three terms on the right hand side of \eqref{4eq:errorSplit} can be
  estimated by the same arguments. We only give details for the first term:
  First note that
  \begin{align*}
    &C_p \frac{1}{h} \Big\| \Big( \sum_{j=1}^{N} \Big| 
    \int_{t_{j-1}}^{t_j} \int_{t_{j-1}}^{t_j} 
	  \int_{t_{j -\frac{1}{2}}}^{t} \big( \dot{g}(s) - \dot{g}(r) \big) \diff{s}
    \diff{r} \diff{W(t)}
    \Big|^2 \Big)^{\frac{1}{2}} \Big \|_{L^p(\Omega_W)}\\
    &\quad = C_p \frac{1}{h}
    \Big( \Big\| \sum_{j=1}^{N} \Big| 
    \int_{t_{j-1}}^{t_j} \int_{t_{j-1}}^{t_j} 
	  \int_{t_{j -\frac{1}{2}}}^{t} 	\big( \dot{g}(s) - \dot{g}(r) \big)
    \diff{s} \diff{r} \diff{W(t)}
    \Big|^2 \Big \|_{L^{\frac{p}{2}}(\Omega_W)}\Big)^{\frac{1}{2}}\\
    &\quad \le C_p \frac{1}{h} \Big( \sum_{j=1}^{N} \Big\| 
    \int_{t_{j-1}}^{t_j} \int_{t_{j-1}}^{t_j} 
	  \int_{t_{j -\frac{1}{2}}}^{t} 	\big( \dot{g}(s) - \dot{g}(r) \big)
    \diff{s} \diff{r} \diff{W(t)}
    \Big\|_{L^{p}(\Omega_W)}^2 \Big)^{\frac{1}{2}}.
  \end{align*}
  Next, we apply Theorem~\ref{th:hoeld} to each summand and obtain
  \begin{align*}
    &\Big( \sum_{j=1}^{N} \Big\| 
    \int_{t_{j-1}}^{t_j} \int_{t_{j-1}}^{t_j} 
	  \int_{t_{j -\frac{1}{2}}}^{t} 	\big( \dot{g}(s) - \dot{g}(r) \big) \diff{s}
    \diff{r} \diff{W(t)}
    \Big\|_{L^{p}(\Omega_W)}^2 \Big)^{\frac{1}{2}}\\
    &\quad \le \Big( \frac{p(p-1)}{2} \Big)^{\frac{1}{2}} h^{\frac{p-2}{2p}} 
    \Big( \sum_{j=1}^{N} \Big( \int_{t_{j-1}}^{t_j} \Big| 
    \int_{t_{j-1}}^{t_j} 
	  \int_{t_{j -\frac{1}{2}}}^{t} \big( \dot{g}(s) - \dot{g}(r) \big) \diff{s}
    \diff{r} \Big|^p \diff{t} \Big)^{\frac{2}{p}} \Big)^{\frac{1}{2}}\\
    &\quad \le \Big( \frac{p(p-1)}{2} \Big)^{\frac{1}{2}} h^{\frac{p-2}{2p}} 
    N^{\frac{p-2}{2p} } 
    \Big( \sum_{j=1}^{N} \int_{t_{j-1}}^{t_j} \Big| \int_{t_{j-1}}^{t_j} 
	  \int_{t_{j -\frac{1}{2}}}^{t} \big( \dot{g}(s) - \dot{g}(r) \big) \diff{s}
    \diff{r} \Big|^p \diff{t}  \Big)^{\frac{1}{p}}\\
    &\quad \le \Big( \frac{p(p-1)}{2} \Big)^{\frac{1}{2}} T^{\frac{p-2}{2p}} 
    \Big( \sum_{j=1}^{N} h^{2(p-1)} 
    \int_{t_{j-1}}^{t_j} \int_{t_{j-1}}^{t_j} 
	  \int_{t_{j -\frac{1}{2}}}^{t} \big| \dot{g}(s) - \dot{g}(r) \big|^p \diff{s}
    \diff{r} \diff{t} \Big)^{\frac{1}{p}}\\
    &\quad \le \Big( \frac{p(p-1)}{2} \Big)^{\frac{1}{2}} T^{\frac{p-2}{2p}} 
    \Big( \sum_{j=1}^{N} h^{2p+p\sigma} 
    \int_{t_{j-1}}^{t_j} \int_{t_{j-1}}^{t_j} 
    \frac{\big| \dot{g}(s) - \dot{g}(r) \big|^p}{|s - r|^{p \sigma + 1}}
	  \diff{s} \diff{r} \Big)^{\frac{1}{p}}\\
    &\quad \le \Big( \frac{p(p-1)}{2} \Big)^{\frac{1}{2}} T^{\frac{p-2}{2p} }
    h^{2 + \sigma} \| g \|_{W^{1+\sigma,p}(0,T)}, 
  \end{align*}
  where we also applied H\"older's inequality several times.
  Thus, together with the factor $C_p \frac{1}{h}$ we arrive at
  \begin{align*}
    &C_p \frac{1}{h} \Big\| \Big( \sum_{j=1}^{N} \Big| 
    \int_{t_{j-1}}^{t_j} \int_{t_{j-1}}^{t_j} 
	  \int_{t_{j -\frac{1}{2}}}^{t} 	\big( \dot{g}(s) - \dot{g}(r) \big) \diff{s}
    \diff{r} \diff{W(t)}
    \Big|^2 \Big)^{\frac{1}{2}} \Big \|_{L^p(\Omega_W)}\\
    &\quad \le C_p \Big( \frac{p(p-1)}{2} \Big)^{\frac{1}{2}}
    T^{\frac{p-2}{2p}} h^{1 + \sigma} \| g \|_{W^{1+\sigma,p}(0,T)}.
  \end{align*}    
  Up to an additional factor $\frac{1}{2}$ the same estimate is valid for the
  other two terms in \eqref{4eq:errorSplit}. This completes the proof.
\end{proof}

\begin{remark}\label{remark:randVars}
  Note that for the implementation of the quadrature rule \eqref{eq:Trap} we
  have to simulate the stochastic integral
  \begin{align*}
  	\int_{t_{j-1}}^{t_j} (t - t_{j -\frac{1}{2}}) \diff{W(t)}
  \end{align*}
  in addition to the standard increments $W(t_j) - W(t_{j-1})$.
  This can easily be accomplished by taking note of
  \begin{align*}
    \E_W \Big[ (W(t_j)- W(t_{j-1}))  \int_{t_{j-1}}^{t_j} (t - t_{j
    -\frac{1}{2}}) \diff{W(t)} \Big]
    &= \int_{t_{j-1}}^{t_j} (t - t_{j -\frac{1}{2}}) \diff{t} = 0,
  \end{align*}
  that is, the two random variables are uncorrelated. 
  Since they are jointly normally distributed, they are also mutually 
  independent. Therefore, we can simulate the two increments in practice
  by generating $(Z_1, Z_2) \sim \mathcal{N}(0, I_2)$ and then setting
  \begin{align*}
    \left(  
    \begin{matrix}
      W(t_j)- W(t_{j-1})\\
      \int_{t_{j-1}}^{t_j} (t - t_{j -\frac{1}{2}}) \diff{W(t)}
    \end{matrix}
    \right)
    \sim
    \left(
    \begin{matrix}
      h^{\frac{1}{2}} & 0 \\
      0 & \frac{1}{2 \sqrt{3}} h^{\frac{3}{2}}
    \end{matrix}
    \right)
    \left(
    \begin{matrix}
      Z_1\\
      Z_2
    \end{matrix}
    \right),
  \end{align*}
  hereby we make use of the fact that 
  \begin{align*}
    \E_W\Big[ \Big| \int_{t_{j-1}}^{t_j} (t - t_{j -\frac{1}{2}}) \diff{W(t)}
    \Big|^2 \Big] 
    = \int_{t_{j-1}}^{t_j} (t - t_{j -\frac{1}{2}})^2 \diff{t} = \frac{1}{12}
    h^3. 
  \end{align*}
\end{remark}


\section{Numerical examples with some deterministic integrands}
\label{sec:num}
 
In this section we perform numerically the quadrature of the It\^{o}-integral
\eqref{eq:stochint} with three deterministic integrands $g_i \colon [0,T]
\to \R$, $i \in \{1,2,3\}$. Hereby, the first integrand $g_1$ is smooth but 
oscillating, while the second is discontinuous with a jump. The third 
integrand is not smooth in the sense that either 
itself or its derivative contains a weak singularity at $t=0$.
We perform a series of numerical experiments which
verify the theoretical results of both quadrature formulas
\eqref{eq:Q} and \eqref{eq:Trap}.

For the implementation of the numerical examples, we follow a similar approach 
as already mentioned in Remark~\ref{remark:randVars}. In order to approximate
the error we simultaneously generate the exact value of the
It\^{o}-integral and the Wiener increments required for the quadrature rules. 
For this we generate a random 
vector $(Z_1,Z_2,Z_3) \sim \mathcal{N}(0,I_3)$ and define
\begin{align}
  \label{eq:simX}
  \left(  
    \begin{matrix}
      X_1\\
      X_2\\
      X_3
    \end{matrix}
  \right)
  :=  	
	\left(  
    \begin{matrix}
      \int_{t_{j-1}}^{t_j} \diff{W(t)}\\
      \int_{t_{j-1}}^{t_j} (t - t_{j -\frac{1}{2}}) \diff{W(t)}\\
      \int_{t_{j-1}}^{t_j} g(t) \diff{W(t)}
    \end{matrix}
    \right)
    \sim
    G
    \left(
    \begin{matrix}
      Z_1\\
      Z_2\\
      Z_3
    \end{matrix}
    \right),
\end{align}
where $t_{j -\frac{1}{2}} = \frac{1}{2}(t_{j-1} + t_j)$ and the matrix $G$ is
the Cholesky decomposition of the covariance matrix $Q \in \R^{3,3}$ given by
\begin{align*}
	Q = \big( \E_W\big[X_n X_m \big] \big)_{n,m \in \{1,2,3\}}.
\end{align*}
Similar to Remark~\ref{remark:randVars} the upper left part of $Q$ takes on the
values
\begin{align*}
	\E_W\big[ X_1^2 \big] = h , \quad
	\E_W\big[ X_2^2 \big] = \frac{h^3}{12},
  \quad \text{and} \quad
	\E_W\big[X_1 X_2 \big] = 0	.
\end{align*}
The newly appearing terms in the third column and row of $Q$ are given by
\begin{align*}
	&\E_W\big[X_3^2 \big] = \int_{t_{j-1}}^{t_j} g^2(t) \diff{t},\quad
	\E_W\big[X_1 X_3 \big] = \int_{t_{j-1}}^{t_j} g(t) \diff{t},
	\quad \text{and} \quad\\
	&\E_W\big[X_2 X_3 \big] = \int_{t_{j-1}}^{t_j} t g(t) \diff{t}
  - t_{j -\frac{1}{2}} \int_{t_{j-1}}^{t_j} g(t) \diff{t}.
\end{align*}
The random variables are then used to compute the exact value of the
It\^{o}-integral as well as the stochastic integral in the higher order 
quadrature formula \eqref{eq:Trap}.
In the same way, we simulate the increments and the exact solution for the
randomly shifted Riemann--Maruyama rule \eqref{eq:Q}, where we do not need to
simulate $X_2$ and 
we have to replace the grid points $\pi_h=(t_j)_{j \in \{0,\ldots,N\}}$ by
those in $\pi_h(\Theta)$ for each realization of the random shift $\Theta \sim
\mathcal{U}(0,1)$ as defined in \eqref{eq:grid2}. For a more detailed
introduction and explanation of this procedure, see, for example,  
\cite[Section~2.3.3]{Glasserman2004}.

In our example we first choose the function $g_1 \colon [0,T] \to \R$ with 
$g_1(t) 
= \sin(\lambda t)$ for a constant value $\lambda \in \R$. For this choice of 
integrand the appearing integrals in the covariance matrix $Q$ can be 
stated explicitly and are given by 
\begin{align*}
  \int_{t_{j-1}}^{t_j} g_1(t) \diff{t} 
  &= \frac{1}{\lambda} \big( - \cos(\lambda t_j) +  \cos(\lambda t_{j-1})
  \big),\\ 
  \int_{t_{j-1}}^{t_j} t g_1(t) \diff{t} 
  &= \frac{1}{\lambda^2} \big( \sin(\lambda t_j) - \sin(\lambda t_{j-1}) \big)
  - \frac{1}{\lambda} \big( t_j \cos(\lambda t_j) - t_{j-1} \cos(\lambda 
  t_{j-1}) \big),
  \intertext{as well as}
  \int_{t_{j-1}}^{t_j} g_1^2(t) \diff{t} 
  &= \frac{h}{2} - \frac{1}{4 \lambda } \big(\sin(2 \lambda t_j) - \sin(2 
  \lambda t_{j-1}) \big).
\end{align*}
Using the fact that $|\sin(t)| \leq t$ holds true 
for all $t\in [0,\infty)$, we obtain for every $h_0 \in (0,T]$ and
$\sigma \in (0,1)$ that
\begin{align*}
  \int_0^{h} \sin^2(\lambda t)  \diff{t}
  \leq \int_0^{h} \lambda^2    t^{2} \diff{t} 
  = \frac{1}{3} \lambda^2 h^2 \quad \text{ for all } h \le h_0.
\end{align*}
Thus, it is easy to see that our choice of the 
integrand $g_1$ fulfills 
Assumption~\ref{as:G} and Assumption~\ref{as:g} for $p =2$ and every 
value $\sigma \in (0,1)$. Therefore,
our results from Theorem~\ref{th:main2} and Theorem~\ref{th:main3}
suggest that the quadrature rule \eqref{eq:Q} converges with a rate of $1$ 
whereas the quadrature rule \eqref{eq:Trap} converges with rate $2$.

Next, for $c \in (0,T)$ we consider the jump function 
\begin{align*}
  g_2 \colon [0,T] \to \R, \quad 
  g_2(t) = 
  \begin{cases}
    0, \quad \text{ if } t \in [0,c),\\
    1, \quad \text{ if } t \in [c,T].
  \end{cases}
\end{align*}
This type of function is considered in more detail in 
Section~\ref{sec:Poisson} coming. There, we prove in 
Lemma~\ref{lem:indicator} that this function is an element of 
$W^{\sigma,p}(0,T)$ for $\sigma p <1$. 
Therefore, Assumption~\ref{as:G} is fulfilled for $p \in [2,\infty)$ and every 
value $\sigma \in \big( 0,\frac{1}{p} \big)$ and
Theorem~\ref{th:main2} yields the convergence of \eqref{eq:Q} with a rate 
$\sigma$.
Note that this function is not even 
continuous, therefore one can not expect to prove any rate of convergence 
when measuring the regularity in an H\"older setting. The integrals 
appearing in the covariance matrix $Q$ can also be stated explicitly as
\begin{align*}
  \int_{t_{j-1}}^{t_j} g_2(t) \diff{t} 
  = \int_{t_{j-1}}^{t_j} g_2^2(t) \diff{t}
  &= 
  \begin{cases}
    0, \quad & \text{ if } t_{j}< c, \\
     t_j - c, \quad & \text{ if } c \in [t_{j-1}, t_j], \\
     t_j - t_{j-1}, \quad & \text{ if } t_{j-1} >c, 
  \end{cases}
  \intertext{and}
  \int_{t_{j-1}}^{t_j} t g_2(t) \diff{t} 
  &= 
  \begin{cases}
    0, \quad &\text{ if } t_{j}< c, \\
    \frac{1}{2} (t_j^2 - c^2), \quad & \text{ if } c \in [t_{j-1}, t_j], \\
    \frac{1}{2}(t_j^2 - t_{j-1}^2), \quad &\text{ if } t_{j-1}>c.
  \end{cases}
\end{align*}

As a third example we consider functions of the form $g_3 \colon [0,T] 
\to \R$
with $g_3(t) = t^{\gamma}$ for $\gamma \in (-\frac{1}{2},\frac{1}{2}] 
\setminus \{0 \}$.
For this choice of integrand the appearing integrals can again be stated
explicitly and are given by  
\begin{align*}
  \int_{t_{j-1}}^{t_j} g_3(t) \diff{t} 
  = \frac{1}{\gamma +1} \big( t_{j}^{\gamma +1} - t_{j-1}^{\gamma +1} \big),
  \quad 
  \int_{t_{j-1}}^{t_j} t g_3(t) \diff{t} 
  = \frac{1}{\gamma +2} \big( t_{j}^{\gamma +2} - t_{j-1}^{\gamma +2} 
  \big),
\end{align*}
as well as
\begin{align*}
  \int_{t_{j-1}}^{t_j} g_3^2(t) \diff{t} 
  = \frac{1}{2 \gamma +1} \big( t_{j}^{2\gamma +1} - t_{j-1}^{2\gamma +1} 
  \big).
\end{align*}
The regularity of the second integrand $g_3$ requires a little more attention 
and depends on the choice of $\gamma$. First, if $\gamma \in (0,\frac{1}{2}]$
the weak derivative of $g_3$ satisfies
$\dot{g_3} \in L^p(0,T)$ for $p< \frac{1}{1 - \gamma}$. Hence,
from Sobolev's embedding theorem, see, for example, \cite[Corollary 
18]{simon1990}, we get
\begin{align*}
  W^{1,p}(0,T) \hookrightarrow W^{\sigma,2}(0,T)
\end{align*}
for $1 - \frac{1}{p} = \sigma - \frac{1}{2}$. This implies $g_3 \in 
W^{\sigma,2}(0,T)$ for every $\sigma 
=  \frac{3}{2} - \frac{1}{p} < \frac{3}{2} - (1 - \gamma)  = \frac{1}{2} + 
\gamma$. Thus, in this case Assumption~\ref{as:G} is satisfied
with $p = 2$ and for all $\sigma \in (0,\frac{1}{2} + \gamma)$
including condition \eqref{eq:G_ini} for the initial value.
Assumption~\ref{as:g} is, however, not satisfied for any value 
$\gamma \in (0,\frac{1}{2}]$. 

Next, we turn to the case $\gamma \in (-\frac{1}{2},0)$, where we explicitly
estimate the Sobolev--Slobodeckij norm. For this 
let $s, t \in [0,T]$ with $s < t$ be arbitrary. 
Then, since $g_3$ is a decreasing, nonnegative function for $\gamma \in
(-\frac{1}{2},0)$ we have
\begin{align*}
  | g_3(t) - g_3(s)| = g_3(s) - g_3(t) \le g_3(s) = s^{\gamma}.
\end{align*}
Moreover, by the fundamental theorem of calculus it holds true that
\begin{align*}
    | g_3(t) - g_3(s)| &= \frac{1}{|\gamma|} \Big| 
    \int_0^1 \big( s + \rho (t-s) \big)^{-1 + \gamma} \diff{\rho} \Big| 
    |t - s|
    \le \frac{1}{|\gamma|} s^{-1 + \gamma} |t-s|.
\end{align*}
Inserting this into the Sobolev--Slobodeckij semi-norm yields for every $\mu
\in (0, \frac{1}{2} + \gamma)$ that
\begin{align*}
  \int_0^T \int_0^T \frac{ |g_3(t) - g_3(s)|^{2} }{|t-s|^{1 + 2 \sigma}}
  \diff{s} \diff{t}
  &= 2 \int_0^T \int_0^t |g_3(t) - g_3(s)|^{2(1- \mu)}
  \frac{ |g_3(t) - g_3(s)|^{2 \mu} }{|t-s|^{1 + 2 \sigma}} 
  \diff{s} \diff{t}\\
  &\le \frac{2}{|\gamma|^{2\mu}} \int_0^T \int_0^t s^{2(1-\mu)\gamma}
  s^{2\mu (-1 + \gamma)} |t-s|^{2\mu - 1 - 2 \sigma} \diff{s} \diff{t}\\
  &= \frac{2}{|\gamma|^{2\mu}} \int_0^T \int_0^t
  s^{2\gamma - 2 \mu } |t-s|^{2\mu - 1 - 2 \sigma} \diff{s} \diff{t}.
\end{align*}
The latter integral is finite for every $\sigma \in (0,\mu)$
due to $2 \gamma - 2 \mu > -1$ by our choice of $\mu \in (0, \frac{1}{2} +
\gamma)$. In sum, this proves that $g_3 \in W^{\sigma,2}(0,T)$ for all
$\sigma \in (0, \frac{1}{2} + \gamma)$. Since condition \eqref{eq:G_ini} is
also easily verified, it again follows that $g_3$ satisfies 
Assumption~\ref{as:G}
with $p= 2$ and for all $\sigma \in (0,\frac{1}{2} + \gamma)$ if $\gamma \in
(-\frac{1}{2},0)$. 
Therefore, we can apply Theorem~\ref{th:main2} and we obtain
that the quadrature rule \eqref{eq:Q} 
converges with a rate of $\gamma + \frac{1}{2}$ in both parameter ranges
$\gamma \in (0,\frac{1}{2})$ and $\gamma \in (-\frac{1}{2},0)$.

Since Assumption~\ref{as:g} is violated for all values of $\gamma$,
Theorem~\ref{th:main3} does not apply to $g_3$. Nevertheless, we still used the
quadrature rule \eqref{eq:Trap} in our numerical experiments 
in this case. 
Hereby, it should be mentioned that for $\gamma \in 
(-\frac{1}{2},0)$ the scheme \eqref{eq:Trap} is actually not well defined, 
since there appears an evaluation of the function $g_3$ at the point $t_0 = 
0$ at which $g_3$ possesses a singularity. 
In the numerical example we made use of the fact, that we knew in advance 
where the singularity is situated and left out this specific summand in the
quadrature rule. 

This problem illustrates well one advantage of a randomized
point evaluation. A quadrature formula based on a deterministic time grid might
not offer a useful approximation if a singularity of the integrand
happens to be at a grid point. On the other hand, an evaluation at a point of a
singularity will not occur almost surely if a randomized grid is used. 


\begin{figure}[t] 
	\centering
	\includegraphics[width=1\textwidth]{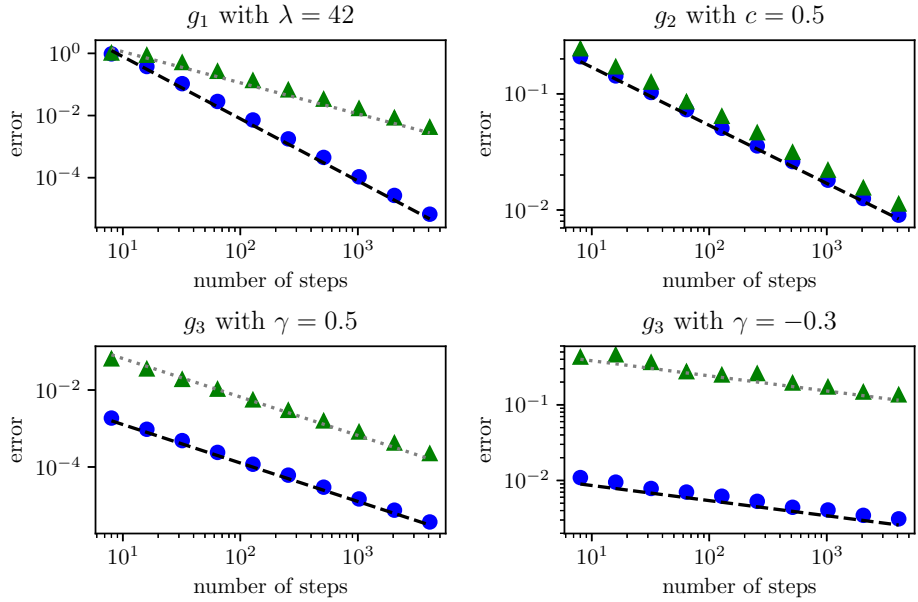}
	\caption{$L^2$-convergence of the lower order scheme \eqref{eq:Q} 
	(green triangles) and 
	the higher order scheme \eqref{eq:Trap} (blue circles) with $g_{1}$ with 
	$\lambda =42$,  
	$g_{2}$ with $c = 0.5$ as well as $g_{3}$ with both $\gamma = 0.5$ and 
	$\gamma = - 0.3$.
  For the function $g_{1}$ we inserted order lines 
  with slopes $1$ and $2$ as well as an order line of slope $0.5$ for $g_2$. In
  the second row we added two order lines with 
  slope $1$ into the left hand subfigure while 
  both order lines have a slope of 0.2 on the right hand side. 
  }  
	\label{fig1}
\end{figure}

\begin{table}[h]
\caption{Numerical example, for $g_2$ with $\gamma = -0.3$}
\label{tab:Numerical_err}
\begin{tabular}{p{1.0cm}|p{1.2cm}p{2.8cm}|p{1.2cm}p{2.8cm}}
  \noalign{\smallskip}\hline\noalign{\smallskip}
  $h$     & error of \eqref{eq:Q} & 95\% conf. interval  for \eqref{eq:Q}, & 
  error of \eqref{eq:Trap} & 95\% conf. interval  for \eqref{eq:Trap} \\ 
  \noalign{\smallskip}\hline\noalign{\smallskip}
  .1250 & .24767 & [.23849, .25652] & .20473 & [.19829, .21097]\\ 
  .0625 & .17613 & [.16952, .18249] & .14520 & [.14053, .14972]\\ 
  .0312 & .12149 & [.11686, .12596] & .10246 & [.09929, .10554]\\ 
  .0156 & .08925 & [.08605, .09234] & .07201 & [.06979, .07417]\\ 
  .0078 & .06480 & [.06226, .06725] & .05062 & [.04901, 
  .05219]\\ 
  .0039 & .04426 & [.04257, .04589] & .03601 & [.03489, 
  .03709]\\ 
  .0020 & .03129 & [.03006, .03248] & .02544 & [.02468, 
  .02618]\\ 
  .0010 & .02205 & [.02122, .02285] & .01809 & [.01750, .01866]\\ 
  .0005 & .01593 & [.01532, .01652] & .01300 & [.01259, .01339]\\ 
  .0002 & .01148 & [.01105, .01190] & .00902 & [.00873, .00930]\\ 
\end{tabular}
\end{table}                 

For the numerical experiment displayed in Figure~\ref{fig1} and Table 
\ref{tab:Numerical_err}, we chose the 
final time $T = 1$ and the parameter values $\lambda = 42$ for $g_1$, $c = 
0.5$ for $g_2$ as well as the parameters $\gamma =-0.3$ and  $\gamma = 
0.5$ for $g_3$. 
As step sizes we took $h_i = 2^{-i}$ with $i \in \{ 3, \dots, 12\}$. For the 
computation of the error we used the sum of the random variables $X_3$
defined in \eqref{eq:simX} as the exact solution. 
For both quadrature formulas, the $L^2(\Omega)$-norm was approximated by
taking the average over $2000$ Monte Carlo iterations. 
The parameter $\theta$ in \eqref{eq:Trap} was chosen to be $0$.

It can be seen in Figure~\ref{fig1} that both quadrature rules \eqref{eq:Q} 
and \eqref{eq:Trap} performed as expected in all our experiments.
In particular, in the case of $g_{1}$
we observed an experimental order of convergence of rate $1$ for \eqref{eq:Q}
and of rate $2$ for \eqref{eq:Trap}. For the function $g_{2}$
the randomly shifted Riemann--Maruyama rule \eqref{eq:Q} converges experimentally  
with a rate of $0.5$. Even though the assumptions for 
Theorem~\ref{th:main3} are not fulfilled, the approximation  
\eqref{eq:Trap} is comparable to \eqref{eq:Q}.
For $g_3$ we expected a convergence rate of $\gamma + \frac{1}{2}$ for 
\eqref{eq:Q} which is well visible in our two numerical tests in the second row
of Figure~\ref{fig1}. Observe that \eqref{eq:Trap} 
shows the same convergence rates in our last two experiments as \eqref{eq:Q} 
but with a better error constant. This indicates that the higher order method is
advantageous even in some situations, where the regularity of the integrand is
not sufficient to ensure a more accurate approximation. However, as already
mentioned above, we had to slightly modify the quadrature rule 
\eqref{eq:Trap}
for $g_3$ with $\gamma = -0.3$ in order to prevent an evaluation of $g_3$ at
its singularity.

To see if the number of 2000 Monte Carlo samples was sufficiently high we also 
computed the $95\%$-confidence intervals based on the central limit
theorem in Table~\ref{tab:Numerical_err}. As one can observe, the
variance of the error estimates are already reasonably small
for both quadrature rules \eqref{eq:Q} and \eqref{eq:Trap} applied to 
$g_2$ with the parameter $\gamma =-0.3$.

\section{Application to Poisson processes}
\label{sec:Poisson}

In this section we apply the randomly shifted Riemann--Maruyama rule
\eqref{eq:Q} for the approximation of a stochastic integral whose
integrand is a Poisson process. To this end, we first 
recall the definition of a Poisson process. Then we show that it
fulfills the condition of Assumption~\ref{as:G}. Finally, we perform a
numerical experiment.

\begin{definition}
  \label{def:Poisson}
  A \emph{Poisson process} $\Pi \colon [0,T] \times \Omega_W \to \N_0$
  with \emph{intensity} $a \in (0, \infty)$ is a stochastic process on
  $(\Omega_W, \F^W, \P^W)$ with the following properties:
  \begin{itemize}
    \item[(i)] There holds $\Pi(0) = 0$ almost surely.
    \item[(ii)] For any $0 \le t_0 < t_1 < \ldots < t_n \le T$, $n \in \N$, the
       random variables $(\Pi(t_i) - \Pi(t_{i-1}))_{i \in
      \{1,\ldots,n\}}$ are independent.
    \item[(iii)] For all $0 \le s \le t \le T$ the law of the 
      increment $\Pi(t) - \Pi(s)$ is the Poisson distribution with mean $a
      (t-s)$, that is
      \begin{align*}
        \P_W\big( \Pi(t) - \Pi(s) = n \big) = \frac{(a(t-s))^n}{n !} \ee^{-a(t-s)},
        \quad \text{ for all } n \in \N_0.
      \end{align*}
    \item[(iv)] The sample paths of $\Pi$ are c\`adl\`ag.
  \end{itemize}
\end{definition}

The following proposition is very useful in order to determine the temporal
regularity of a typical sample path of a Poisson process. A proof is found, for
instance, in \cite[Proposition~4.9]{peszat2007}.

\begin{prop}
  \label{prop:Poisson}
  Let $\Pi \colon [0,T] \times \Omega_W \to \N_0$ be a Poisson process with
  intensity $a \in (0, \infty)$. Then there exists an independent and
  with the same parameter $a \in (0,\infty)$ exponentially distributed family 
  of random
  variables $(Z_n)_{n \in \N}$ on $(\Omega_W,\F^W,\P_W)$ such that 
  \begin{align}
    \label{eq:repPoisson}
    \Pi(t) =
    \begin{cases}
      0, & \text{if } t \in [0, Z_1),\\
      k, & \text{if } t \in [Z_1 + \ldots + Z_k, Z_1 + \ldots + Z_{k + 1}).
    \end{cases}
  \end{align}
\end{prop}

We recall that a random variable $Z \colon \Omega_W \to \R$ is
exponentially distributed with parameter $a \in (0,\infty)$ if
\begin{align*}
  \P_W ( Z > x ) = \ee^{-a x} \quad \text{ for all } x \in [0,\infty).
\end{align*}

Next, let us introduce an indicator function $I_c \colon
 [0,T] \to \R$, $c \in [0,\infty)$,
of the form $I_c(t) = \one_{[c,\infty)}(t)$, $t \in [0,T]$.
It then follows from Proposition~\ref{prop:Poisson} that we can formally
write $\Pi$ as a series of the form
\begin{align}
  \label{eq:series1}
  \Pi(t,\omega) = \sum_{k = 1}^{\infty} I_{S_k(\omega)}(t), \quad t
  \in [0,T], \, \omega \in \Omega_W,
\end{align}
where the random jump points $S_k(\omega)$ are given by
\begin{align}
  \label{eq:sumZ}
  S_k(\omega) := \sum_{j = 1}^k Z_j(\omega), \quad \text{for all } \omega \in
  \Omega_W.
\end{align}
The following lemma is concerned with the temporal regularity of the indicator
function $I_c$, $c \in [0,\infty)$.

\begin{lemma}
  \label{lem:indicator}
  For every $c \in [0,T]$, $\sigma \in (0,1)$, and $p \in [1,\infty)$ with
  $\sigma p < 1$ it holds true that $I_c \in W^{\sigma,p}(0,T)$. In addition,
  we have
  \begin{align*}
    \sup_{c \in [0,T]} \| I_c \|_{W^{\sigma,p}(0,T)} < \infty.
  \end{align*}
\end{lemma}

\begin{proof}
  Since the indicator function is bounded by $1$ we directly get
  \begin{align*}
    \| I_c \|_{L^p(0,T)} \le T^{\frac{1}{p}}
  \end{align*}
  for all $p \in [1,\infty)$. In addition, for every $c \in [0,T]$, $\sigma \in
  (0,1)$, and $p \in [1,\infty)$ with $\sigma p < 1$ we have
  \begin{align*}
    &\int_0^T \int_0^T \frac{|I_c(t) - I_c(s)|^p}{|t-s|^{1 + \sigma p}} \diff{t}
    \diff{s}\\
    &\quad = \int_0^c \int_c^T \frac{1}{|t-s|^{1 + \sigma p}} \diff{t}
    \diff{s}
    + \int_c^T \int_0^c \frac{1}{|t-s|^{1 + \sigma p}} \diff{t}
    \diff{s}\\
    &\quad = \frac{2}{\sigma p} \int_c^T \big( (t - c)^{-\sigma p} - t^{-\sigma
    p} \big) \diff{t}
    \le \frac{2}{\sigma p(1 - \sigma p)} T^{1 - \sigma p}. 
  \end{align*}
  Since $c \in [0,T]$ was arbitrary, the assertion follows.
\end{proof}

We are now well-prepared to verify that every Poisson process indeed satisfies
the conditions of Assumption~\ref{as:G}.

\begin{theorem}
  \label{cor:regPoisson}
  Let $\Pi \colon [0,T] \times \Omega_W \to \N_0$ be a Poisson process with
  intensity $a \in (0, \infty)$.  Then, for any $p \in [2, \infty)$, $\sigma
  \in (0,1)$ with $\sigma p < 1$ we have 
  \begin{align*}
    \Pi \in L^p( \Omega_W; W^{\sigma,p}(0,T) ).
  \end{align*}
  In addition, for every $h_0 \in (0,T]$ there exists $C_0 \in (0,\infty)$ such 
  that 
  \begin{align}
    \label{eq:Poisson_ini}
    \int_0^{h} \E_W\big[ |\Pi(t)|^p \big] \diff{t} \le C_0 h^{\max(0,p\sigma - 
      \frac{p-2}{2}) } \quad \text{for all } h \leq h_0.
  \end{align}
  In particular, every Poisson process with intensity $a \in (0,\infty)$
  fulfills the conditions of Assumption~\ref{as:G} for every $p \in
  [2,\infty)$ and $\sigma \in (0,1)$ with $\sigma p < 1$.
\end{theorem}

\begin{proof}
  First, let $p \in [1,\infty)$ be arbitrary. We observe that a
  typical sample path of $\Pi$ is nonnegative and increasing.
  Hence, we have $\sup_{t \in [0,T]} \| \Pi(t) \|_{L^p(\Omega_W)} 
  = \| \Pi(T) \|_{L^p(\Omega_W)} < \infty$ by the Poisson distribution
  of $\Pi(T)$ with mean $aT$. From this we immediately obtain
  \begin{align*}
    \int_0^{h} \E_W\big[ |\Pi(t)|^p \big] \diff{t} \le C_0 h
  \end{align*}
  for all $h \leq h_0$. Since $\max(0,p\sigma -  \frac{p-2}{2}) < 1$ 
  for $p \in [2,\infty)$ and $\sigma p < 1$ condition \eqref{eq:Poisson_ini}
  follows. 

  Furthermore, we obtain $\P_W(A) = 1$
  where $A \in \F^W$ denotes the event
  \begin{align*}
    A = \big\{ \omega \in \Omega_W \, : \, \sup_{t \in [0,T]} \Pi(t,\omega) =
    \Pi(T,\omega) <
    \infty \big\}.
  \end{align*}
  Then, for every $\omega \in A$ the series in \eqref{eq:series1} 
  consists in fact of only finitely many indicator functions.
  More precisely, there exists $N(\omega) := \Pi(T,\omega) \in \N_0$ such that
  \begin{align}
    \label{eq:series}
    \Pi(t,\omega) = \sum_{k = 1}^{N(\omega)} I_{S_k(\omega)}(t), \quad t
    \in [0,T],
  \end{align}
  where $S_k(\omega)$ are defined in \eqref{eq:sumZ}. 
  Together with Lemma~\ref{lem:indicator} this proves that for every $p \in 
  [1,\infty)$, $\sigma \in (0,1)$ with $\sigma p < 1$ we have
  \begin{align}
    \label{eq:pathreg}
    \P_W \big( 
    \{ \omega \in \Omega_W\, : \, \Pi(\cdot, \omega) \in W^{\sigma,p}(
    0,T) \} \big) = 1.
  \end{align}
  Hence, it remains to show that
  \begin{align*}
    \E_W \Big[ \int_0^T \int_0^T \frac{|\Pi(t) - \Pi(s)|^p}{|t - s
    |^{1 + \sigma p} } \diff{s} \diff{t} \Big] < \infty.
  \end{align*}
  To this end, we insert the representation \eqref{eq:series} 
  and obtain
  \begin{align*}
    &\E_W \Big[ \int_0^T \int_0^T \frac{|\Pi(t) - \Pi(s)|^p}{|t - s
    |^{1 + \sigma p} } \diff{s} \diff{t} \Big]\\
    &\quad = \sum_{n = 0}^\infty \int_{\Omega_W} \one_{\{\Pi(T,\omega) =
    n\}}(\omega) 
    \int_0^T \int_0^T \frac{|\Pi(t,\omega) - \Pi(s,\omega)|^p}{|t - s
    |^{1 + \sigma p} } \diff{s} \diff{t} \diff{\P_W(\omega)}\\
    &\quad \le \sum_{n = 0}^\infty \sum_{k = 1}^n
    \int_{\Omega_W} \one_{\{\Pi(T,\omega) = n\}}(\omega) n^{p-1}
    \int_0^T \int_0^T \frac{ | I_{S_k(\omega)}(t) -
    I_{S_k(\omega)}(s) |^p}{|t - s
    |^{1 + \sigma p} } \diff{s} \diff{t} \diff{\P_W(\omega)}\\
    &\quad \le \sum_{n = 0}^\infty \sum_{k = 1}^n
    \int_{\Omega_W} \one_{\{\Pi(T,\omega) = n\}}(\omega) n^{p-1}
    \big\| I_{S_k(\omega)} \big\|^p_{W^{\sigma,p}(0,T)} 
    \diff{\P_W(\omega)}\\
    &\quad \le \sup_{c \in [0,T]} \big\| I_{c} \big\|^p_{W^{\sigma,p}(0,T)} 
    \sum_{n = 0}^\infty n^{p} \int_{\Omega_W} \one_{\{\Pi(T,\omega) =
    n\}}(\omega) \diff{\P_W(\omega)}\\
    &\quad \le \sup_{c \in [0,T]} \big\| I_{c} \big\|^p_{W^{\sigma,p}(0,T)} 
    \big\| \Pi(T) \big\|_{L^p(\Omega_W)}^p,
  \end{align*}
  where we also used that $S_k(\omega) \in [0,T]$ for all $\omega \in
  \{ \Pi(T) = n \}$ and $1 \le k \le n$.
  An application of Lemma~\ref{lem:indicator} then completes the proof.
\end{proof}

\begin{figure}[t] 
	\centering
	\includegraphics[width=0.7\textwidth]{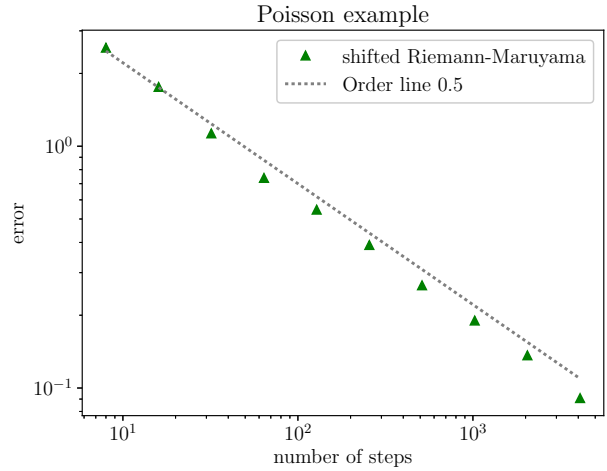}
	\caption{$L^2$-convergence of the lower order scheme \eqref{eq:Q} to the
  It\^{o}-integral of a Poisson process with intensity $a = \frac{3}{4}$ on
  the interval $[0,10]$ with $2000$ Monte Carlo samples. 
  }
  \label{fig2}
\end{figure}

\begin{table}[h]
  \caption{Numerical example, for Poisson process}
  \label{tab:Poisson_L2_err}
  \begin{tabular}{p{1.5cm}p{1.5cm}p{1.2cm}p{4cm}}
     \noalign{\smallskip}\hline\noalign{\smallskip}
     $h$     & error & EOC & 95\% conf. interval \\ 
     \noalign{\smallskip}\hline\noalign{\smallskip}
      1.2500      & 2.55293 &  & [2.45273, 2.64935]\\ 
      0.6250  & 1.65424 & 0.63 & [1.58914, 1.71688]\\ 
      0.3125  & 1.12986 & 0.55 & [1.08814, 1.17010]\\ 
      0.1562  & 0.76850 & 0.56 & [0.73918, 0.79675]\\ 
      0.0781  & 0.54830 & 0.49 & [0.52936, 0.56660]\\ 
      0.0391  & 0.37698 & 0.54 & [0.36380, 0.38971]\\ 
      0.0195  & 0.26343 & 0.52 & [0.25427, 0.27227]\\ 
      0.0098  & 0.17800 & 0.57 & [0.17186, 0.18394]\\ 
      0.0049  & 0.12968 & 0.46 & [0.12501, 0.13419]\\ 
  \end{tabular}
\end{table}

We close this section with a short numerical experiment. Hereby we
applied the randomly shifted Riemann--Maruyama quadrature rule for the
approximation of an It\^{o}-integral whose integrand is a Poisson process. 
For
the error plot displayed in Figure~\ref{fig2} we chose the final time $T=10$
and the intensity parameter $a = \frac{3}{4}$. As step sizes we took 
$h \in \{ T\, 2^{-i}\, : \, i = 3,\ldots,11\}$. For the approximation of the
error we compared the result of the quadrature rule 
with a given step size $h$ to a numerical reference solution
with the smaller step size $\frac{h}{16}$ driven by the same stochastic
trajectories. In addition, the $L^2(\Omega)$-norm was approximated by a
standard Monte Carlo simulation with $2000$ independent samples.

As one can see in Figure~\ref{fig2}, the randomly shifted Riemann--Maruyama
rule performed as expected with an experimental order of convergence close to
$\frac{1}{2}$, in agreement with the regularity of the Poisson process. Since
we already knew from Section~\ref{sec:num} that the higher order quadrature
rule \eqref{eq:Trap} does not yield an advantage if the integrand has jumps, it
was not implemented in this example. In Table~\ref{tab:Poisson_L2_err} we also
show the numerical values of the computed errors and corresponding
asymptotically valid  $95\%$-confidence intervals based on the central limit
theorem. Apparently, already with just $2000$ Monte Carlo samples the variance
of the error estimator is quite decent. 

\section*{Acknowledgement}

The authors wish to express their gratitude to Stefan Heinrich for many
interesting discussions on this topic.
This research was carried out in the framework of \textsc{Matheon}
supported by Einstein Foundation Berlin. The second named author also
gratefully acknowledges financial support by the German Research Foundation
through the research unit FOR 2402 -- Rough paths, stochastic partial
differential equations and related topics -- at TU Berlin. 


\end{document}